\documentclass[hidelinks,onefignum,onetabnum]{siamart220329}
\newsiamthm{assump}{Assumption}
\crefname{subsection}{section}{sections}
\Crefname{subsection}{Section}{Sections}
\crefname{section}{Section}{Sections}
\Crefname{section}{Section}{Sections}
\crefrangeformat{figure}{Figures~#3#1#4 and~#5#2#6}
\Crefrangeformat{figure}{Figures~#3#1#4 through~#5#2#6}
\crefrangeformat{section}{Sections~#3#1#4 and~#5#2#6}
\Crefrangeformat{section}{Sections~#3#1#4 through~#5#2#6}
\crefrangeformat{subsection}{Sections~#3#1#4 and~#5#2#6}
\Crefrangeformat{subsection}{Sections~#3#1#4 through~#5#2#6}
\crefrangeformat{equation}{equations~(#3#1#4) through~(#5#2#6)}
\Crefrangeformat{equation}{Equations~(#3#1#4) through~(#5#2#6)}

\usepackage{lipsum}
\usepackage{amsfonts}
\usepackage{graphicx}
\usepackage{epstopdf}
\usepackage{algorithmic}
\usepackage{upgreek}
\usepackage{dsfont}
\ifpdf
  \DeclareGraphicsExtensions{.eps,.pdf,.png,.jpg}
\else
  \DeclareGraphicsExtensions{.eps}
\fi

\newsiamremark{remark}{Remark}
\newsiamremark{hypothesis}{Hypothesis}
\crefname{hypothesis}{Hypothesis}{Hypotheses}
\newsiamthm{claim}{Claim}

\headers{Global regularization of 3D volume integral operators}{T. G. Anderson, M. Bonnet, L. M. Faria, and C. P\'erez-Arancibia}

\title{A general-purpose global regularization method for 3D volume integral operators\thanks{Submitted to the editors \today.
}}

\author{Thomas G. Anderson\thanks{Department of Computational Applied Mathematics \& Operations Research, Rice University, Houston, TX USA
  (\email{thomas.anderson@rice.edu}).}
\and Marc Bonnet\thanks{POEMS (CNRS, INRIA, ENSTA), ENSTA Paris, 91120 Palaiseau, France
  (\email{marc.bonnet@ensta-paris.fr}, \email{luiz.faria@ensta-paris.fr}).}
\and Luiz M. Faria\footnotemark[3]
\and Carlos~P\'erez-Arancibia\thanks{Department of Applied Mathematics, University of Twente, Enschede, The Netherlands (\email{c.a.perezarancibia@utwente.nl}).}
}

\usepackage{amsopn}

\ifpdf
\hypersetup{
  pdftitle={Global regularization of general 3D volume integral operators},
  pdfauthor={T. G. Anderson, M. Bonnet, L. M. Faria, and C. P\'erez-Arancibia}
}
\fi

\newcommand{\Newcommand}[2]{\providecommand{#1}{}\renewcommand{#1}{#2}}

\makeatletter
\newcommand{\proofstep}[1]{%
  \par
  \addvspace{\smallskipamount}
  \textit{#1\@addpunct{.}}\enspace\ignorespaces
}
\makeatother

\newlength{\kaka}
\newlength{\LSpace}

\providecommand{\ahref}[2]{}

\newcommand{\shtimes}{\hspace*{-0.1em}\times\hspace*{-0.1em}}

\newcommand{\shm}{\hspace*{-0.1em}-\hspace*{-0.1em}}
\newcommand{\shp}{\hspace*{-0.1em}+\hspace*{-0.1em}}
\newcommand{\sheq}{\hspace*{-0.1em}=\hspace*{-0.1em}}

\newcommand{\shg}{\hspace*{-0.1em}>\hspace*{-0.1em}}

\newcommand{\shin}{\hspace*{-0.1em}\in\hspace*{-0.1em}}

\newcommand{\shsetm}{\hspace*{-0.1em}\setminus\hspace*{-0.1em}}
\newcommand{\shsubs}{\hspace*{-0.1em}\subset\hspace*{-0.1em}}

\newcommand{\del}[1][]{\partial_{#1}}

\newcommand{\vv}[1]{\stackrel{\scriptscriptstyle\vee}{#1}\!\!\rule{0em}{1ex}}

\newcommand{\inv}[1]{\dfrac{1}{#1}}

\newcommand{\jump}[1]{[\![ {#1} ]\!]}

\newcommand{\lsqb}{\big[\hspace*{0.1em}}
\newcommand{\rsqb}{\hspace*{0.1em}\big]}
\newcommand{\lcb}{\big\{\hspace*{0.1em}}
\newcommand{\rcb}{\hspace*{0.1em}\big\}}
\newcommand{\Lcb}{\Big\{\hspace*{0.1em}}
\newcommand{\Rcb}{\hspace*{0.1em}\Big\}}
\newcommand{\Lsqb}{\Big[\hspace*{0.1em}}
\newcommand{\Rsqb}{\hspace*{0.1em}\Big]}
\newcommand{\lpar}{\big(\hspace*{0.1em}}
\newcommand{\rpar}{\hspace*{0.1em}\big)}

\newcommand{\labs}{\big|\hspace*{0.1em}}
\newcommand{\rabs}{\hspace*{0.1em}\big|}

\newcommand{\lnrm}{\big\|\hspace*{0.1em}}
\newcommand{\rnrm}{\hspace*{0.1em}\big\|}

\newcommand{\Lpar}{\Big(\,}
\newcommand{\Rpar}{\,\Big)}

\newcounter{refcounter}
\usepackage{ifthen}

\newcommand{\de}{\delta \eta}

\newcommand{\OO}{\Omega}

\newcommand{\G}{\Gamma}

\Newcommand{\Sp}{S_{p}}

\newcommand{\ds}{\,\text{d}s}

\newcommand{\dv}{\,\text{d}v}

\newcommand{\dx}{\,\text{d}x}
\newcommand{\dy}{\,\text{d}y}

\newcommand{\iO}{\int_{\OO}}

\newcommand{\iG} {\int_{\G}}

 \newlength{\meno}
 \newlength{\integrale}
 \newlength{\trattino}
 \newlength{\integralegrande}
 \newlength{\uguale}
 \newlength{\ugualegrande}

\newcommand{\Bcal}{\mathcal{B}}
\newcommand{\Ccal}{\mathcal{C}}
\newcommand{\Dcal}{\mathcal{D}}
\newcommand{\Ecal}{\mathcal{E}}

\newcommand{\Ical}{\mathcal{I}}

\newcommand{\Lcal}{\mathcal{L}}

\newcommand{\Ncal}{\mathcal{N}}
\newcommand{\Ocal}{\mathcal{O}}
\newcommand{\Pcal}{\mathcal{P}}
\newcommand{\Qcal}{\mathcal{Q}}
\newcommand{\Rcal}{\mathcal{R}}
\newcommand{\Scal}{\mathcal{S}}
\newcommand{\Tcal}{\mathcal{T}}

\newcommand{\Vcal}{\mathcal{V}}
\newcommand{\Wcal}{\mathcal{W}}
\newcommand{\Xcal}{\mathcal{X}}

\newcommand{\Zcal}{\mathcal{Z}}

\Newcommand{\PS}{\mbox{\boldmath $\mathcal{P}$}}

\Newcommand{\NC} {{\rm NC}}

\Newcommand{\NF}{\rmN\Fsub}

\Newcommand{\NN}{\rmN\Nsub}

\newcommand{\dotp}{\raisebox{1pt}{\hspace*{1pt}\scalebox{0.45}{$\bullet$}}\hspace*{1pt}}

\newcommand{\eps}{\varepsilon}

\newcommand{\oo}{\omega}

\newcommand{\Hsup}{^{\text{\tiny H}}}

\newcommand{\Tsup}{^{\text{\tiny T}}}

\newcommand{\Fsub}{_{\text{\tiny F}}}

\newcommand{\Nsub}{_{\text{\tiny N}}}

\newcommand{\inc}{_{\text{inc}}}

\Newcommand{\Ref}{^{\text{ref}}}

\Newcommand{\skew}{_{\text{skew}}}

\Newcommand{\Tr}{\text{Tr}}

\renewcommand{\Re}{\text{Re}}

\newcommand{\rmi} {\mathrm{i}}

\newcommand{\rmN} {\mathrm{N}}

\newcommand{\rmQ} {\mathrm{Q}}

\newcommand{\sfA} {\mathsf{A}}

\newcommand{\sfC} {\mathsf{C}}

\newcommand{\sfK} {\mathsf{K}}

\newcommand{\sfv} {\mathsf{v}}

\newcommand{\Khat}{\hat{K}}

\newcommand{\yHat}{\widehat{y}}

\newcommand{\KHat}{\widehat{K}}

\newcommand{\Lbar}{\Lbar{}{}}

\newcommand{\Nbar}{\Nbar{}{}}
\newcommand{\Obar}{\Obar{}{}}

\newcommand{\Rbar}{\Rbar{}{}}

\newcommand{\Tbar}{\Tbar{}{}}

\Newcommand{\hbar}{\bar{h}{}}

\Newcommand{\Bbb} {\mathbb{B}}
\newcommand{\Cbb} {\mathbb{C}}

\newcommand{\Nbb} {\mathbb{N}}

\newcommand{\Rbb} {\mathbb{R}}

\newcommand{\OmegaB}{\overline{\Omega}}

\Newcommand{\bfeta}{\boldsymbol{\eta}}

\Newcommand{\de}{\,\mathrm{d}}

\Newcommand{\inc}{\mathrm{inc}}

\newcommand{\p}{\partial}

\newcommand{\R}{\mathbb{R}}
\newcommand{\C}{\mathbb{C}}
\newcommand{\N}{\mathbb{N}}

\Newcommand{\dv}{\operatorname{div}}
\newcommand{\GG}{\mathsf{G}}
\newcommand{\Id}{\operatorname{I}}

\renewcommand{\vv}[1]{\boldsymbol{#1}}
\usepackage{amsmath}

\newcommand{\subT}{_{\mathscr{T}}}
\newcommand{\subL}{_{\mathscr{L}}}
\newcommand{\BcalT}{\widetilde{\Bcal}}

\newcommand{\sfKBr}{\widebreve{\sfK}}
\newcommand{\ZcalB}{\widebreve{\Zcal}{}}

\usepackage{paralist}
\usepackage{multirow}
\usepackage{amsmath}
\usepackage{amsfonts}
\usepackage{amsmath}
\usepackage{amssymb}
\usepackage{graphicx}
\usepackage{mathrsfs}
\usepackage{upgreek}
\usepackage{booktabs}
\usepackage{cite}
\usepackage{hyperref}
\usepackage{mathtools}
\usepackage{enumitem}
\usepackage[inkscapearea=page]{svg}
\usepackage{adjustbox}
\usepackage{nccmath}

\makeatletter
\def\widebreve{\mathpalette\wide@breve}
\def\wide@breve#1#2{\sbox\z@{$#1#2$}%
     \mathop{\vbox{\m@th\ialign{##\crcr
\kern0.08em\brevefill#1{0.8\wd\z@}\crcr\noalign{\nointerlineskip}%
                    $\hss#1#2\hss$\crcr}}}\limits}
\def\brevefill#1#2{$\m@th\sbox\tw@{$#1($}%
  \hss\resizebox{#2}{\wd\tw@}{\rotatebox[origin=c]{90}{\upshape(}}\hss$}
\makeatletter

\makeatletter
\let\save@mathaccent\mathaccent
\newcommand*\if@single[3]{%
  \setbox0\hbox{${\mathaccent"0362{#1}}^H$}%
  \setbox2\hbox{${\mathaccent"0362{\kern0pt#1}}^H$}%
  \ifdim\ht0=\ht2 #3\else #2\fi
  }
\newcommand*\rel@kern[1]{\kern#1\dimexpr\macc@kerna}
\newcommand*\widebar[1]{\@ifnextchar^{{\wide@bar{#1}{0}}}{\wide@bar{#1}{1}}}
\newcommand*\wide@bar[2]{\if@single{#1}{\wide@bar@{#1}{#2}{1}}{\wide@bar@{#1}{#2}{2}}}
\newcommand*\wide@bar@[3]{%
  \begingroup
  \def\mathaccent##1##2{%
    \let\mathaccent\save@mathaccent
    \if#32 \let\macc@nucleus\first@char \fi
    \setbox\z@\hbox{$\macc@style{\macc@nucleus}_{}$}%
    \setbox\tw@\hbox{$\macc@style{\macc@nucleus}{}_{}$}%
    \dimen@\wd\tw@
    \advance\dimen@-\wd\z@
    \divide\dimen@ 3
    \@tempdima\wd\tw@
    \advance\@tempdima-\scriptspace
    \divide\@tempdima 10
    \advance\dimen@-\@tempdima
    \ifdim\dimen@>\z@ \dimen@0pt\fi
    \rel@kern{0.6}\kern-\dimen@
    \if#31
      \overline{\rel@kern{-0.6}\kern\dimen@\macc@nucleus\rel@kern{0.4}\kern\dimen@}%
      \advance\dimen@0.4\dimexpr\macc@kerna
      \let\final@kern#2%
      \ifdim\dimen@<\z@ \let\final@kern1\fi
      \if\final@kern1 \kern-\dimen@\fi
    \else
      \overline{\rel@kern{-0.6}\kern\dimen@#1}%
    \fi
  }%
  \macc@depth\@ne
  \let\math@bgroup\@empty \let\math@egroup\macc@set@skewchar
  \mathsurround\z@ \frozen@everymath{\mathgroup\macc@group\relax}%
  \macc@set@skewchar\relax
  \let\mathaccentV\macc@nested@a
  \if#31
    \macc@nested@a\relax111{#1}%
  \else
    \def\gobble@till@marker##1\endmarker{}%
    \futurelet\first@char\gobble@till@marker#1\endmarker
    \ifcat\noexpand\first@char A\else
      \def\first@char{}%
    \fi
    \macc@nested@a\relax111{\first@char}%
  \fi
  \endgroup
}
\makeatother

\usepackage[normalem]{ulem}

\definecolor{burntumber}{rgb}{0.54, 0.2, 0.14}

\usepackage[format=hang,font=small,labelfont={bf,sl},textfont=sl,width=0.96\textwidth]{caption}
\usepackage{subcaption}
\graphicspath{{./Figures/}{./}}

\begin{document}
\maketitle

\begin{abstract}
  Singular volume integral operators associated with constant-coefficient
  partial differential operators extend the applicability of potential theory
  to inhomogeneous problems, for example arising from nonlinearities or
  variable coefficients. Typically the PDE kernels in these operators give rise to singularities at all $\Ocal(1/h^3)$ volume
  discretization/evaluation points in a mesh of characteristic size $h$, while the slowly-decaying nature of such kernels give rise to long-range interactions that require
  coupling to fast summation algorithms. The presented method
  uses Green's identities to regularize a wide variety of both scalar-valued and vector-valued volume integral operators by use of a certain regularizing volume density interpolant. The analysis shows how the
  regularizing effect of the interpolant is \emph{global} in the sense that the interpolation
  quality increases in an exactly compensatory fashion as the distance to the
  Green's function singularity decreases. High-order convergence estimates with
  tabulated simplex quadratures are established, including with exact
  representation of curved domains.
\end{abstract}

\begin{keywords}
  volume potential, volume integral operator, integral equations, high-order quadrature, curved
  meshing, fast algorithm
\end{keywords}

\begin{MSCcodes}
  65R20, 65D32
\end{MSCcodes}

\maketitle 

\section{Introduction}
This paper presents a regularization scheme for singular volume integral operators (VIOs) of the form
\begin{equation}\label{VIO:Z}
\Zcal [f](x) \coloneqq \int_{\Omega} \sfK(x,y) f(y)\de y,\quad x\in\R^3,
\end{equation}
where the bounded domain $\Omega \subset \R^3$ has a piecewise-smooth boundary $\G$, the (scalar-, vector- or tensor-valued) kernel $\sfK$ is defined in terms of a free-space Green's function $\GG$ associated with a constant-coefficient partial differential operator (PDO) $\Lcal$ and (possibly) its derivatives, and $f \in C^s(\overline\Omega,\C^\sigma)$ ($s\geq 0$ to be specified later) is a given field (with scalar, vector or tensor values depending on the VIO considered and the underlying PDE) that is termed the source density. The archetypal VIO is the Newton potential operator, for which $\sfK=\GG$, but the proposed methodology generalizes significantly to kernels for which the second Green's identity can be used on resulting potentials. Concrete PDE models that fit the framework of this paper include for example, but are not limited to, operators associated with Laplace, Helmholtz (including with complex parameter), isotropic as well as anisotropic elasticity (and their time-harmonic counterparts), and Stokes flows.

For an overview of numerical methods for evaluating VIOs in 2D we refer the reader to~\cite{Anderson:22a,anderson2024fast}; in what follows we describe only what is relevant to the 3D problem.
When the domain is a finite union of cubes the Parallel Volume FMM (PVFMM) software~\cite{malhotra2014volume,malhotra2015pvfmm} accurately and efficiently computes
volume potentials; in the context of general geometry
problems it has been used in conjunction with penalty
methods~\cite{malhotra2014volume} which generally leads to first-order accurate
solutions in the $L^2(\Omega)$ norm.  Much less has been done for the 3D
problem in general complex geometries, and available works~\cite{Steinbach:10,mohyaddin2023fast} have not
successfully demonstrated high-order
accuracy:  reference~\cite{Steinbach:10} does establish second-order convergence in energy or $L^2$ norms on the boundary $\G$, whereas while~\cite{mohyaddin2023fast} claims high-order convergence, in the norm of $L^2(\Omega)$ but not in any stronger norm, only first-order convergence is demonstrated.

The present work is closely linked to the authors' prior work~\cite{anderson2024fast} which, like this paper, proposed volume density interpolation methods to regularize volume integrals, but which, unlike this work, was limited to the Laplace and Helmholtz operators in 2D. This work extends these methods in a fashion generic to all strongly elliptic operators, establishes their convergence in three dimensions, treats additional integral operators involving more-singular kernels, as well as  additionally develops improved regularization schemes even for those operators treated prior. Finally, while the convergence analysis of the prior work was limited to affine elements, the present work treats \emph{isogeometric} (also referred to as \emph{exact}) elements which reproduce without geometric error the precise domain under consideration, as we describe next.

The domain $\Omega$, with curved boundary assumed to be piecewise-$C^\theta$, can be discretized into curved domain elements which coincide with or approximate this curved boundary. An exact curved domain element is one such that the boundary of the element coincides with the domain boundary; in practice, as finite element methods involve polynomial test spaces, the final methods use polynomial approximation of the functions defining the domain elements leading to for example \emph{isoparametric} finite element methods, where the same space of polynomials is used for approximating each of the solution and geometry. Constructing curved 3D elements, either approximate or exact, is highly nontrivial~\cite{lenoir1986optimal,bernardi1989optimal,dubois1990discrete}, whereas the analysis of resulting finite element methods involves not only that of function discretization error but also of error emanating from domain approximation, see e.g.~\cite{zlamal1973curved,zlamal1974curved}. Our interest here, however, lays in Nystr\"om integral equation methods which do not involve polynomial test spaces; consequently we find it natural to retain geometrically exact element parametrizations that thus obviate the need for analysis of the latter error.  The present paper uses \emph{exact} curved elements which are a tunably-smooth perturbation of a straight domain element; the main technical property of the curved mappings generating such elements that our analysis relies upon is that the mappings are, in the language of~\cite{bernardi1989optimal}, `regular of order $\theta$'---for an element of characteristic size $h$, derivatives of order $m \le \theta+1$ of the map are $\mathcal{O}(h^m)$.

The paper is organized as follows. The preliminary \Cref{sec:prelim} introduces a class of PDOs treated by the methodology and several types of VIOs of relevance. \Cref{sec:vdim} presents the integral regularization strategy, describes the curved element mappings to be utilized and summarizes necessary (known) theoretical results concerning these mappings. \Cref{sec:analysis} presents the error analysis of resulting regularized volume integrals. Finally, numerical results are reported in \Cref{sec:numer} and some summary remarks conclude the paper in \Cref{sec:conclusion}.\enlargethispage*{1ex}

\section{Preliminaries}
\label{sec:prelim}

\subsection*{Strongly-elliptic PDOs}

To allow sufficient generality while retaining definiteness, our main focus is on VIOs associated with strongly elliptic (second-order, constant-coefficient) PDOs $\Lcal$, which (using definitions and notation style of~\cite[Chaps.~4, 6]{mclean}) are of the generic form
\begin{equation}
  \Lcal u = \medop\sum_{j=1}^d \lsqb -\del[j]\Bcal_j u + A_j\del[j]u \rsqb + Au, \qquad
  \Bcal_j u := \medop\sum_{k=1}^d A_{jk}\del[k]u. \label{PDO}
\end{equation}
where $d$ is the physical space dimension (the primary focus of this work being $d=3$).

They act on a vector-valued function $u:\OO\to\Cbb^p$ to give a vector-valued function $\Lcal u:\OO\to\Cbb^p$, and their coefficients $A_{jk},A_j,A$ are, for each value of coordinate indices $j,k$, complex $p\shtimes p$ matrices. The strong ellipticity assumption on $\Lcal$ consists in the principal (i.e. leading-order) part $\Lcal_0$ of $\Lcal$, defined as $\Lcal_0 u := -\sum_{j=1}^d \del[j]\Bcal_j u$, verifying
\begin{equation}
  \Re\Lcb \medop\sum_{j=1}^d\medop\sum_{k=1}^d \lsqb A_{jk}\xi_k\eta \rsqb\Hsup \xi_j\eta \Rcb \geq C|\xi|^2|\eta|^2 \quad\text{for all }\xi\shin\Rbb^d,\,\eta\shin\Cbb^p,
\end{equation}
where $()\Hsup$ denotes the Hermitian transpose of a vector or matrix. For scalar problems ($p=1$), $A_{jk},A_j,A$ are scalars and the strong ellipticity assumption reduces to
\begin{equation}
  \Re\Lcb \medop\sum_{j=1}^d\medop\sum_{k=1}^d \overline{A_{jk}\xi_k}\, \xi_j \Rcb \geq C|\xi|^2 \quad\text{for all }\xi\shin\Rbb^d.
\end{equation}
However, our proposed numerical treatment of VIOs additionally applies in various cases not covered by the strong ellipticity assumption, which notably comprise the Stokes and Maxwell systems, see the end of the present Section. (The specific property that a PDO $\Lcal$ must possess is that the second Green's identity~\cref{green3}, discussed shortly, holds, as it does for strongly elliptic operators.)\enlargethispage*{12ex}

The conormal derivative $\Bcal_{\nu}u$ of $u$ is defined by
\begin{equation}
  \Bcal_{\nu}u := \medop\sum_{j=1}^d \nu_j\gamma\lpar \Bcal_j u \rpar,
\end{equation}
where $\nu$ is the outward unit normal to $\G$, $\gamma$ is the trace operator for $\OO$ and $\G=\del\OO$ is the boundary of $\OO$. For example, $\Bcal_{\nu}u=\del[\nu]u$ is the normal derivative of $u$ for the Laplace or Helmholtz equations, and produces the traction vector in linear elasticity.

The second Green identity associated with the PDO $\Lcal$ then reads
\begin{equation}
  \iO \lsqb v\Hsup(\Lcal u) - (\Lcal^{\star}v)\Hsup u \rsqb \dx
 = \iG \lsqb (\BcalT_{\nu}v)\Hsup\gamma u - (\gamma v)\Hsup\Bcal_{\nu}u \rsqb \ds,
\label{green3}
\end{equation}
with the formal adjoint PDO $\Lcal^{\star}$ and the adjoint conormal derivative $\BcalT_{\nu}u$ defined by
\begin{equation}
  \Lcal^{\star} u = -\medop\sum_{j=1}^d \del[j]\BcalT_j A\Hsup u, \quad
  \BcalT_j u := \medop\sum_{k=1}^d A\Hsup_{kj}\del[k]u + A\Hsup_j u, \qquad
  \BcalT_{\nu}u := \medop\sum_{j=1}^d \nu_j\gamma\lpar \BcalT_j u \rpar.
\end{equation}

The PDO $\Lcal$ has a full-space Green's (scalar- or tensor-valued) function $\GG$ verifying
\begin{equation}
  \Lcal_y \GG(x,y) = \delta(y-x)\Id_p \label{Green}
\end{equation}
together with decay or radiation conditions (depending on the PDO considered) at infinity; moreover, $\GG$ is translation-invariant, i.e. has the form
\begin{equation}
  \GG(x,y) = \Upphi(x-y), \qquad x,y\in\R^3 \label{nabla:G}
\end{equation}
for some function $\Upphi$. A Green's function $\GG^{\star}(x,y)$ for the formal adjoint PDO $\Lcal^{\star}$ may be defined in the same way, and identity~\eqref{green3} for $\OO$ containing an arbitrarily large ball then allows to show that
\begin{equation}
  \GG^{\star}(x,y) = \GG\Hsup(y,x), \qquad x,y\in\R^3. \label{G:Gstar}
\end{equation}
For any strongly-elliptic PDO $\Lcal$ ($d = 3$), $\GG(x,y)$ has a $\Ocal(|y-x|^{-1})$ singularity at the source point $x$; in fact, the Green's function $\GG_0$ associated with the principal part $\Lcal_0$ is positively homogeneous with degree $-1$ in $y-x$, see e.g.~\cite[Thm.~6.8]{mclean}, while the complementary part $\GG-\GG_0$ remains bounded at $y=x$.\enlargethispage*{10ex}

\subsection*{Examples of PDOs and their Green's functions} The foregoing assumptions on $\Lcal$ cover many classical PDOs of mathematical physics. They comprise the scalar ($p=1$) Laplace, Helmholtz and anisotropic Laplace PDOs
\begin{equation}
  \text{(a) \ }\Lcal u = -\Delta u, \qquad \text{(b) \ }\Lcal u = -\Delta u - (\oo^2/c^2)u, \qquad
  \text{(c) \ }\Lcal u = -\dv(\sfA\nabla u) \label{laplace}
\end{equation}
(with $A_{ik}=\delta_{ik}$ for (a,b), $A_{ik}$ the entries of a symmetric positive definite matrix $\sfA$ for (c), and $A=-\oo^2/c^2$ for (b)) as well as their vector versions. They also include the steady-state and time-harmonic linear isotropic elasticity PDOs ($p=d$)
\begin{equation}
  \Lcal u = -\mu \Delta u-(\mu+\lambda) \nabla(\dv u), \quad
  \Lcal u = -\mu \Delta u-(\mu+\lambda) \nabla(\dv u) -\omega^2 \rho u \label{elast}
\end{equation}
governing a displacement field $u$ ($\rho$ denoting the mass density, and $\lambda,\mu$ the Lam\'e parameters, of the medium), as well as their extensions to anisotropic elasticity (where $[A_{ik}]_{ab}=\sfC_{aibk}$ for a real-valued fourth-order elasticity tensor $\sfC$ having the usual ellipticity, minor symmetry and major symmetry properties). For the time-harmonic equations, $\omega>0$ denotes the angular frequency. Another example, for which (unlike in the aforementioned cases) $\Lcal^{\star}\not=\Lcal$, is the scalar diffusion-advection equation
\begin{equation}
  \Lcal u = -\dv(\sfA\nabla u) + \sfv\cdot\nabla u \label{adv:dif}
\end{equation}
(with $\sfA$ as in~(\ref{laplace}c) and $A_i=\sfv_i$) where $\sfv\in\R^3$ is a known spatially-uniform velocity.

In general, the Green's function verifying~\eqref{Green} (together with the suitable decay or radiation condition at infinity) is known as an inverse spatial Fourier transform, as exemplified by general anisotropic elasticity. Well-known closed-form expressions of $\GG(x,y)=\Upphi(x-y)$ are however available in a number of cases. In particular, focusing on the three-dimensional case ($d=3$), we have
\begin{equation}
  \text{(a) \ }\Upphi(z) = \inv{4\pi|z|}, \qquad
  \text{(b) \ }\Upphi(z) = \frac{e^{\rmi\oo|z|/c}}{4\pi|z|}, \qquad
  \text{(c) \ }\Upphi(z) = \inv{4\pi\sqrt{\text{det}\sfA}|z|_{\sfA}} \label{laplace:G}
\end{equation}
for~\eqref{laplace} (with $|z|_{\sfA}:=\sqrt{z\Tsup\sfA^{-1}z}$ for (c)). In the linear isotropic elasticity case~\eqref{elast}, we have
\begin{equation}
  \Upphi(z):= \begin{cases}
\displaystyle \inv{8\pi\mu|z|} \Lcb \frac{\lambda+3 \mu}{\lambda+2 \mu}\Id_p+\frac{\lambda+\mu}{\lambda+2 \mu} \frac{zz\Tsup}{|z|^2} \Rcb & (\text{elastostatics}) \\[0.3cm]
\displaystyle\inv{4\pi\mu|z|}\Lcb A(|z|) \Id_p+B(|z|) zz\Tsup \Rcb & (\text{elastodynamics})\end{cases}
\end{equation}
with, for the isotropic elastodynamic case, the auxiliary functions $A$ and $B$ given by
\begin{equation}
\begin{aligned}
  A(s) &:= \Lpar 1+\frac{i}{k_T s}-\frac{1}{k_T^2 s^2} \Rpar \mathrm{e}^{\rmi k_T s}
  - \frac{k_L^2}{k_T^2} \Lpar \frac{\rmi}{k_L s}-\frac{1}{k_L^2 s^2} \Rpar \mathrm{e}^{\rmi k_L s} ,\\
  B(s) &:= \Lpar \frac{3}{k_T^2 s^2}-\frac{3\rmi}{k_T s}-1 \Rpar \mathrm{e}^{\rmi k_T s}
  - \frac{k_L^2}{k_T^2} \Lpar \frac{3}{k_L^2 s^2}-\frac{3\rmi}{k_L s}-1 \Rpar \mathrm{e}^{\rmi k_L s}.
\end{aligned}
\end{equation}
where $k_L^2:=\rho\omega^2/(\lambda+2 \mu)$ and $k^2_T:=\rho \omega^2/\mu$ are the wavenumbers of compressive (a.k.a. longitudinal) and shear (a.k.a. transversal) elastic waves, respectively. As for the Green's function for the diffusion-advection PDO~\eqref{adv:dif}, it is given by
\begin{equation}
  \Upphi(z) = \frac{\exp\lpar \sfv\Tsup\sfA^{-1}z-|\sfv|_{\sfA}\,|z|_{\sfA}\rpar}{4\pi\sqrt{\text{det}\sfA}|z|_{\sfA}}.
   \label{adv:dif:G}
\end{equation}

\subsection*{Volume potentials and their uses} We focus in this work on three classical types of VIOs, presented along increasing orders of kernel singularity. The first one is the Newton potential operator
\begin{equation}
  \Vcal[f](x) := \iO \GG(x,y) f(y) \dy,\qquad x \in \R^d, \label{V:def}
\end{equation}
for a source density $f:\OO\to\Cbb^p$ (implicitly extended by zero to $\R^d \setminus \overline{\Omega}$)), which generates particular solutions to inhomogeneous PDEs as it verifies~\cite[Ch.\ 6]{mclean} (see for more ~\cite[Ch.\ 6, Thm.\ III]{kel})
\begin{equation}
  \Lcal\Vcal[f] = f \qquad\text{in }\R^d \setminus \Gamma, \qquad \jump{\Vcal[f]} = \jump{\Bcal_\nu \Vcal[f]} = 0 \quad \text{on } \Gamma, \label{V:PDE}
\end{equation}
where $\jump{\cdot}$ is the jump over $\Gamma$. The gradient $\nabla \Vcal$ is also frequently of interest.

Next, for densities $g:\OO\to\Cbb^{p\times d}$, we define the VIO
\begin{equation}
  \Wcal[g](x) := \dv\Lpar \iO \GG(x,y) g(y) \dy \Rpar = -\iO \nabla_y \GG(x,y) \cdot g(y) \dy,\qquad x \in \R^d, \label{W:def}
\end{equation}
(with the second equality resulting from~\eqref{nabla:G} and the "$\cdot$" symbol indicating a tensor contraction over all indices of the lowest-order tensor, here $g$), which verifies 
\begin{equation}
  \Lcal\Wcal[g] = \dv g \quad\text{in }\OO\cup(\R^d\shsetm\OmegaB), \qquad
  \jump{\Bcal_{\nu}\Wcal[g]}=g\nu \quad \text{on }\G, \label{W:PDE}
\end{equation}
i.e. it generates particular solutions to inhomogeneous PDEs with divergence-form right-hand sides; such solutions for example occur when studying the response of elastic bodies subjected to thermal strains or residual stresses. The first relation in~\cref{W:PDE} can be established e.g. using $\Lcal \Wcal = \dv \Lcal \Vcal$ and~\cref{V:PDE} (or using the divergence theorem, see~\cref{W:IPP}), while the jump relation follows from the divergence theorem applied to~\cref{W:def}, see~\cref{W:IPP}, followed by use of the jump relation~\cite[Thm.\ 6.11]{mclean} for $\Bcal_{\nu} \Scal$ and the fact that $\jump{\Bcal_\nu \Vcal} = 0$.

In addition to providing particular PDE solutions, the VIOs $\Vcal$ and $\Wcal$ are involved in volume integro-differential equations (VIEs) governing fields in media whose physical parameters deviate within a bounded region $\OO$ from reference constant values entering the PDO $\Lcal$. Let the perturbed medium be described by the PDO $\Lcal'$, of the form~\eqref{PDO} but with possibly space-dependent coefficients $A'_{jk},A'_{j},A$ whose support is contained in $\overline{\OO}$ so that $\delta\Lcal:=\Lcal'-\Lcal=0$ outside of $\overline{\OO}$. Letting $u$ be a given background field satisfying $\Lcal u=0$ in $\R^d$, one seeks $u'$ solving $\Lcal' u'$ in $\R^d$ and such that $u'-u$ is decaying or radiating (depending on the physics) at infinity. Postulating the ansatz $u'-u=\Vcal[f]+\Wcal[g]$ and using~\eqref{V:PDE}, \eqref{W:PDE} together with the requirements $\Lcal(u'-u)=-\delta\Lcal u'$ and $\jump{\Bcal_{\nu}(u'\shm u)}=\delta\Bcal_{\nu}u'\mid_{-}$ (implied by $\Lcal u=0$ and $\Lcal'u'=0$) then yields $f,g$ in terms of $u'$ and $\delta A_{jk},\,\delta A_{j},\,\delta A$, so that $u'$ is finally found after some algebra to satisfy the VIE\enlargethispage*{7ex}
\begin{equation}
  u'(x)
  - \Wcal\Lsqb \medop\sum_{j=1}^d \lpar \delta\Bcal_i u' \rpar e_j\Tsup \Rsqb(x)
  + \Vcal\Lsqb \medop\sum_{j=1}^d (\delta A_j)\del[j]u' + (\delta A)u' \Rsqb(x) = u(x), \; x\in\R^d. \label{VIE}
\end{equation}

If $\delta A_{jk}=0$ and $\delta A_{j}=0$ (i.e. if medium perturbations modify only the lowest-order part of $\Lcal$), \eqref{VIE} reduces to the classical Lippmann-Schwinger VIE
\begin{equation}
  u'(x) + \Vcal\lsqb (\delta A)u' \rsqb(x) = u(x), \qquad x\in\R^d,
\end{equation}
associated with e.g. wave scattering by perturbations of refraction index or (in elasticity) mass density. Alternatively, if $\delta A_j=0$ and $\delta A=0$ (i.e. if medium perturbations modify only the principal part $\Lcal_0$ of $\Lcal$), applying the operator $\sum_{i=1}^d\delta\Bcal_i e_i\Tsup$ to~\eqref{VIE} produces the VIE
\begin{equation}
  g(x) - \medop\sum_{i=1}^d \lpar \delta\Bcal_i \Wcal[g](x) \rpar e_i\Tsup
 = \medop\sum_{i=1}^d (\delta\Bcal_i u) e_i\Tsup, \qquad x\in\R^d,
\end{equation}
whose unknown is $g:= \sum_{j=1}^d (\delta\Bcal_j u')e_j\Tsup$. This type of singular VIE appears in theoretical and computational studies of heterogeneous conducting (or elastic) media, as it directly relates the flux (or stress) perturbation induced by a medium perturbation to a given background potential gradient (or strain), the singular VIO $\sum_{j=1}^d \lpar \delta\Bcal_i \Wcal[g](x) \rpar e_i\Tsup = \sum_{j=1}^d \lpar \sum_{k=1}^d \delta A_{ik}\del[k]\Wcal[g](x) \rpar e_i\Tsup$ being sometimes referred to as the Green (or $\Gamma$) operator in this context~\cite{willis:77}. This motivates us to also make our treatment applicable to the singular VIO
\begin{equation}
  \Xcal[g](x) := \nabla\Wcal[g](x),\qquad x \in \R^d, \label{X:def}
\end{equation}
whose explicit expression
\begin{equation}
  \Xcal[g](x)
 = S\cdot g(x) - \text{PV}\iO \nabla_x\lpar \nabla_y\GG(x,y)\cdot g(y) \rpar \dy \label{X:expr}
\end{equation}
(with the tensor contraction symbol "$\cdot$" as in~\eqref{W:def}) involves a Cauchy principal value (PV) volume integral (i.e the limit as $\eps\to0$ of the corresponding integral over $\OO\setminus B_{\eps}(x)$) and the tensor $S$ defined, for any (vector or tensor depending on the underlying PDE) constant $b$, by
\begin{equation}
  S\cdot b := -\lim_{\eps\to0} \int_{\partial B_{\eps}(x)} \nu(y) \lpar \nabla_y\GG(x,y)\cdot b \rpar\Tsup \ds(y) \label{S:def}
\end{equation}
(the limit being non-zero due to the $|y-x|^{-2}$ singularity of $\nabla_y\GG(x,y)$). For example, $S$ is given by $S=-\tfrac{1}{3}I_3$ in the case of Laplace, and is known in terms of the (4th-order) Eshelby tensor for the unit ball~\cite{mura} in the case of linear elasticity. As the discussion leading to~\eqref{aux01} indicates, $S$ will in practice not be involved in the proposed numerical evaluation method applied to the VIO $\Xcal$. We also note that the above PV integral exists whenever $G$ is a Green function.

\subsection*{Additional application cases}
In some cases, the governing PDO for the relevant physical model is not strongly elliptic, yet associated volume potentials are still defined in terms of usual (e.g. Laplace or Helmholtz) Green's functions and derivatives thereof. Our methodology is then also often valid in such cases: it in particular applies to volume potentials associated with the Stokes flow and Maxwell systems, since the requisite polynomial PDE solutions are available for these cases using the construction methods of~\cite{anderson2024construction} and Green's identities for these PDEs are available~\cite[\S 2.3, \S 6.9]{hsiao:wend:21}. As another example, \cite{costabel:12} provides a governing VIE for the three-dimensional scattering of a given incident electromagnetic wave propagating in vacuum by a compactly-supported penetrable object $\OO$ whose isotropic permittivity and permeability are arbitrary constants (in particular allowed to take complex values), with the governing VIO expressed as a linear constant-coefficient combination of potentials $\Vcal,\Wcal,\Xcal$ involving the scalar Helmholtz kernel.\enlargethispage*{1ex}

\section{Volume Density Interpolation}\label{sec:vdim}

In line with our previous treatment of 2D cases~\cite{anderson2024fast}, our proposed method for evaluating a generic VIO $\Zcal$ on a given density $f$ assumes a tessellation $\Tcal_h$ of $\Omega$ and is concerned with the evaluation of $\Zcal[f](x)$ for $x$ spanning a finite mesh-dependent set $\Ecal_h$ of evaluation points in $\OO$, with the integration in~\eqref{VIO:Z} treated as a sum of integrals over the elements of $\Tcal_h$. Let $K$ denote the element that contains a given evaluation point $x$, the corresponding element integral being thus singular. Our strategy relies on additively splitting $\Zcal[f](x)$ (for $x\in K$) as
\begin{equation}
  \Zcal[f](x)
 = \Zcal\lsqb f-f_n(\cdot;K) \rsqb(x) + \Zcal\lsqb f_n(\cdot;K) \rsqb(x), \label{Z:split}
\end{equation}
$f_n(\cdot;K)$ being the polynomial of total degree $n$ (of same tensor order as $f$) that interpolates $f$ on a discrete set $\Ical_n \sheq \{ x_j: x_j\shin K\}_{j=1}^{q_n}$ of $q_n$ points within $K$, where
\begin{equation}
  q_n := \text{dim}(\Pcal_n) = (n+1)(n+2)(n+3)/6 \label{qn:expr}
\end{equation}
is the dimension of the space $\Pcal_n$ of all tri-variate polynomials of total degree at most~$n$. The purpose of a decomposition of the form~\eqref{Z:split} is to make $f-f_n(\cdot;K)$ ``small enough'' on $K$ to allow neglecting the (potentially-singular) contribution of $K$ to $\Zcal\lsqb f-f_n(\cdot;K) \rsqb(x)$ in the approximate evaluation process for $\Zcal[f](x)$. This formally amounts to replacing therein the operator $\Zcal$ and its kernel $\sfK$ by ``punctured'' versions
\begin{equation}\label{eq:punctured_green_function}
  \ZcalB[f](x) := \iO \sfKBr(x,y) f(y) \dy, \qquad
  \sfKBr(x,y) := \biggl\{\ \begin{aligned} &\sfK(x,y), &&x\in K, y\notin K, \\[-0.5ex] &0,&& x\in K,y\in K. \end{aligned}
\end{equation}
(which are by construction nonsingular) and writing
\begin{equation}
  \Zcal\lsqb f-f_n(\cdot;K) \rsqb(x) \approx \ZcalB[f](x) - \ZcalB[f_n(\cdot;K)](x). \label{Z:punct}
\end{equation}
The burden of singular integration is then transferred by~\eqref{Z:split} to the task of evaluating $\Zcal\lsqb f_n(\cdot;K) \rsqb(x)$ which, by using a suitably-chosen polynomial PDE solution associated to $f_n$, will be analytically recast as a boundary integral. The latter is nonsingular for any evaluation point interior to $\OO$, so that this part of the treatment effects an indirect regularization. We now describe in more detail the proposed evaluation method for given $K\in\Tcal_h$ and $x\in K$, where $f_n$ and related polynomials implicitly depend on $K$.\enlargethispage*{3ex}

\subsection{Evaluation of $\Zcal[f_n](x)$}\label{eval:Z:reg}

The treatment for this part depends in its details on which potential is being evaluated. Beginning with the Newtonian operator $\Zcal=\Vcal$, we let $\Phi_n=\Phi_n(\cdot;K)$ be a (possibly vector- or tensor-valued) polynomial satisfying
\begin{equation}\label{eq:reg_function}
  \Lcal\Phi_n = f_n \quad\text{in}\ \R^3.
\end{equation}
A systematic approach for the practical construction of such polynomial PDE solutions for any given $f_n$ is developed in~\cite{anderson2024construction} and yields solutions $\Phi_n$ of degree $n\shp r$, where $r$ is the derivative order of the lowest-order part of the PDO $\Lcal$ (e.g. $r\sheq 2$ for Laplace, $r\sheq 1$ for diffusion-convection, $r\sheq 0$ for Helmholtz); moreover, $\Phi_n$ is unique if $r\sheq 0$~\cite{anderson2024construction}. Writing the Green identity~\eqref{green3} for $u\sheq \Phi_n$ and $v\sheq G^{\star}(\cdot,x)\sheq \GG\Hsup(x,\cdot)$ yields
\begin{equation}
  \Vcal[f_n](x) = \mu(x)\Phi_n(x) + \Dcal[\Phi_n](x) - \Scal[\Bcal_{\nu}\Phi_n](x) \label{V:IPP}
\end{equation}
(with $\mu(x)=\Id_p$ if $x\shin\OO$, $\mu(x)=0 \Id_p$ if $x\shin\R^3\shsetm\overline{\OO}$, and $\mu(x)=\gamma(x)$ if $x\shin\G$ with the tensor $\gamma(x)$ depending only on the arrangement of the tangent planes to $\G$ at $x$), whereby $\Vcal[f_n]$ is recast in terms of boundary integrals through the single-layer and double-layer potentials~\cite{mclean}
\begin{equation}\label{eq:single_double_layers}
  \Scal[\varphi](x) := \iG \GG(x,y) \varphi(y) \ds(y), \quad
  \Dcal[\varphi](x) := \iG \lsqb \BcalT_{\nu}\GG\Hsup(x,y) \rsqb\Hsup \varphi(y) \ds(y).
\end{equation}
For the gradient of $\Vcal$ which arises naturally in Neumann as well as in nonlinear boundary value problems, one has, for $x \not\in \Gamma$,
\begin{equation}\label{Vgrad:IPP}
    \nabla \Vcal[f_n](x) = \mu(x)\nabla\Phi_n(x) + \nabla \Scal[\Phi_n](x) - \nabla \Dcal[\Bcal_{\nu}\Phi_n](x),
\end{equation}
an expression which can be used as a regularizer for the evaluation of $\nabla \Vcal$.

To reformulate $\Wcal[g_n]$ in a similar way ($g_n=g_n(\cdot; K)$ being a $p$-component-wise polynomial interpolant of $g$), we apply the divergence theorem to its defining integral~\eqref{W:def}, yielding
\begin{equation}
  \Wcal[g_n](x) = -\Scal[g_n\cdot\nu](x) + \Vcal[\dv g_n](x), \label{W:IPP}
\end{equation}
and then use~\eqref{V:IPP} with $f_n=\dv g_n$. Letting $\Psi_n=\Psi_n(\cdot,K)$ be a polynomial satisfying
\begin{equation}\label{eq:reg_function:div}
\Lcal\Psi_n = \dv g_n \quad\text{in}\ \R^3,
\end{equation}
we obtain~(cf. the more singular kernels in~\cite[eq.~(2.5)]{anderson2024fast})
\begin{equation}
  \Wcal[g_n](x)
  = \mu(x)\Psi_n(x) + \Dcal[\Psi_n](x) - \Scal\lsqb \Bcal_{\nu}\Psi_n + g_n\cdot\nu \rsqb(x),
\label{W:IPP2}
\end{equation}

Finally, to deal with the evaluation of $\Xcal[g_n](x) = \nabla_x\Wcal[g_n](x)$ on a polynomial interpolant $g_n$ of $g$, we apply $\nabla_x$ to the integrated-by-parts expression~\eqref{W:IPP} of $\Wcal[g_n]$, exploit~\eqref{nabla:G} and apply the divergence theorem to the (weakly singular) volume integral, to obtain
\begin{align}
  \Xcal[g_n](x)
 &= -\nabla_x\Scal[g_n\cdot\nu](x) - \iO \nabla_y \GG(x,y) (\dv g_n)(y) \dy \\
 &= -\nabla_x\Scal[g_n\cdot\nu](x) - \Scal\lsqb(\dv g_n)\nu\Tsup\rsqb(x) + \Vcal\lsqb\nabla\dv g_n\rsqb(x). \label{X:g_n} \notag
\end{align}
Proceeding as before, we then use~\eqref{V:IPP} with $f_n=\nabla\dv g_n$, letting $\Upsilon_n=\Upsilon_n(\cdot,K)$ be a (vector or tensor) polynomial satisfying
\begin{equation}\label{eq:reg_function:graddiv}
\Lcal\Upsilon_n = \nabla\dv g_n \quad\text{in}\ \R^3,
\end{equation}
($\Upsilon_n=\nabla\Psi_n$ being one such polynomial solution) to finally find
\begin{equation}
  \Xcal[g_n](x)
 = \mu(x)\Upsilon_n(x) - \nabla_x\Scal[g_n\cdot\nu](x) - \Scal\lsqb(\dv g_n)\nu\Tsup + \Bcal_{\nu}\Upsilon_n\rsqb(x) + \Dcal\lsqb \Upsilon_n \rsqb(x).
\label{X:IPP}
\end{equation}

In summary, the regularizing term $\Zcal[f_n](x)$ is found to take the generic form
\begin{equation}
  \Zcal[f_n](x) = \mu(x)\Xi_n(x) + \Bcal_{\Zcal}[f_n,\Xi_n](x) \label{Z:reg}
\end{equation}
for all four cases $\Zcal=\Vcal,\nabla \Vcal, \Wcal,\Xcal$, where $\Xi_n=\Phi_n, \nabla \Phi_n, \Psi_n,\Upsilon_n$ is the relevant polynomial PDE solution satisfying~\eqref{eq:reg_function},  \eqref{eq:reg_function:div}, or~\eqref{eq:reg_function:graddiv} and $\Bcal_{\Zcal}[f_n,\Xi_n](x)$ gathers all boundary potentials involved in~\eqref{V:IPP}, \eqref{Vgrad:IPP}, \eqref{W:IPP} or~\eqref{X:IPP}.\enlargethispage*{10ex}

\subsection{Evaluation of $\protect\ZcalB[f\shm f_n](x)$}

The proposed ``puncturing approximation'' \eqref{Z:punct} features, as will be shown, nonsingular integrals, and thus only requires a quadrature rule for smooth functions defined on the elements $K_\ell\shin\Tcal_h$, $\ell\sheq1,\ldots,L$. Volume integrals of sufficiently regular functions $\phi$ can be expressed via the sum
\begin{equation}\label{eq:tri_smooth_integral_decomp}
  \int_\Omega \phi(y) \dy = \sum_{\ell=1}^L \int_{\KHat} \phi({T}_\ell({\yHat})) \left|\operatorname{det} \mathsf J_\ell({\yHat})\right|\de\yHat,
\end{equation}
over all $L$ elements $K_{\ell}$ of $\Tcal_h$, where each element $K_{\ell}$ is defined via a mapping $T_\ell:\KHat\to K_{\ell}$ from the reference tetrahedron $\KHat$ and $\mathsf J_\ell$ is the Jacobian matrix of $T_\ell$. Quadrature for each of the integrals over $\KHat$ in~\cref{eq:tri_smooth_integral_decomp} is performed in this work via a 3D extension~\cite{gimbutas:25} of the Vioreanu-Rokhlin quadrature method~\cite{Vioreanu:14}, which uses $q_n$ quadrature nodes $\yHat_{j} \in \KHat$ (called 3D-VR nodes thereafter) and positive weights $\widehat{\omega}_j$; in the sequel we denote this quadrature rule applied to a function $\psi: \KHat \to \C$ as $\widehat{\mathcal{Q}}^m[\psi]$. Such quadrature rule is capable of exactly integrating polynomials of total degree up to $m=m(n)>n$ (some values of $m$ are tabulated in Table~\ref{3D-VR:m}). Moreover, since $q_n=\text{dim}(\Pcal_n)$ (see~\eqref{qn:expr}), 3D-VR nodes also conveniently function as interpolation nodes for constructing $n$-th degree density interpolants. Letting $N = L q_n$, real-space nodes are obtained via the mappings $\{T_\ell\}_{\ell=1}^L$ yielding the node-weight set
\begin{equation}\label{eq:globnodesweights}
    \upvartheta := \{(\xi_r, \upomega_r)\}_{r=1}^N
        = \bigcup_{\ell = 1}^L \Lcb \left({T}_\ell(\yHat_j),\ \widehat{\omega}_j \,\labs \operatorname{det} {\mathsf J}_\ell(\yHat_j)|\right) \Rcb_{j=1}^{q_n}
\end{equation}
of global quadrature nodes and weights, and the quadrature approximation of~\eqref{eq:tri_smooth_integral_decomp} can thus be written as
\begin{equation}\label{eq:standard_QR}
  \int_\Omega \phi(y)\de y \approx
  \Qcal^{h,m}_\Omega[\phi] \coloneqq \sum_{j\in\rmQ(\OO)} \upomega_j \phi(\vv{\xi}_j) = \sum_{j=1}^N\upomega_j\phi(\vv\xi_j),
\end{equation}
where the set $\rmQ(E):=\lcb j\,\labs (\vv{\xi}_j,\upomega_j)\shin\upvartheta,\vv{\xi}_j\shin E\,\rcb$ collects the indices of all quadrature nodes lying in a subset $E$ of $\OO$. The notations set in~\eqref{eq:standard_QR} allow to refer to quadrature approximations $\Qcal_{E}$ of integrals over mesh-conforming subsets $E\shsubs\Omega$.
With the foregoing definitions, the proposed quadrature approximation of $\ZcalB[f\shm f_n]$ is
\begin{multline}
  \ZcalB^{h,m}[f\shm f_n(\dotp,K)]
 := \Qcal^{h,m}_\Omega\lsqb \sfKBr(x,\dotp) [f - f_n(\dotp; K)] \rsqb \\
 = \Qcal^{h,m}_{\Omega\setminus K}\lsqb \sfKBr(x,\dotp) [f - f_n(\dotp; K)] \rsqb
\end{multline}

The error analysis in~\cite{anderson2024fast} for the evaluation of quantities analogous  to $\ZcalB[f-f_n](x)$ relies on   quadrature estimates for approximating the left-hand side
of~\cref{eq:standard_QR} by $\mathcal{Q}_\Omega[\phi]$ where, for each $x$, $\phi(y) = \sfKBr(x, y) [f(y) - f_n(y; K)]$. Allowing curvilinear transformations obstructs this, as the classical composite estimates~\cite{IsaacsonKeller} require the ability to integrate exactly polynomials of a certain degree, a property that does not hold when $T_\ell$ is not affine as $\left|\operatorname{det} \mathsf J_\ell\right|$ is then not a constant. This work instead develops estimates for $D^\alpha \phi(T_\ell(\hat{y}))$ and subsequent high-order quadrature error estimates for the approximation~\cref{eq:standard_QR} by means of classical estimates for the quadrature rule $\widehat{\mathcal{Q}}^m$. The essential property for high-order accuracy is that $\|D^\alpha T_\ell\|_2 \lesssim h^{|\alpha|}$. The maps $T_\ell$ themselves are described and their properties are established in the following section; suitable composite error estimates are developed in \Cref{sec:analysis}.\enlargethispage*{1ex}

\begin{table}[t]
\caption{Degree of exactness $m$ for 3D-VR quadrature rules~\cite{gimbutas:25} using $q_n$ nodes, with $q_n$ given by~\eqref{qn:expr}.\label{3D-VR:m}}\centering
\begin{tabular}{|c|ccccccccccc|}\hline
 $n$ & 0 & 1 & 2 & 3 & 4 & 5 & 6 & 7 & 8 & 9 & 10 \\ \hline
 $m$ & 1 & 2 & 3 & 5 & 6 & 7 & 9 & 10 & 11 & 13 & 15 \\ \hline
\end{tabular}
\end{table}

\subsection{Curved-element mesh construction for piecewise-smooth domains}
The method and convergence analysis relies on elementary quadratures over domains discretized with straight and curved elements, that is, using a family $\Tcal$ of meshes $\Tcal_h$ that satisfies the usual quality constraints~\cite{Sauter2010}. Throughout this section, $h_K$ will denote the diameter of an element $K$ and $\rho_K$ will denote the supremum of the diameter of spheres contained within $K$.
\begin{definition}\label{def:regularuniform}
For a family of tetrahedralizations of $\Omega$, or mesh family, denoted by $\Tcal = \{ \Tcal_h \}_{h>0}$ with $h>0$ denoting the maximum element size in $\Tcal_h$, let the shape-regularity constant $\kappa_{\Tcal}$ and the quasi-uniformity constant $q_\Tcal$ of $\Tcal$ be defined by
\begin{equation}\label{eq:shaperegulardef:quasiuniformdef}
  \kappa_{\Tcal} := \sup_{h>0}\max_{K\in\Tcal_h}\frac{h_K}{\rho_K}, \qquad
        q_\Tcal := \sup_{h>0}\frac{h}{\min_{K\in\Tcal_h}h_K}.
\end{equation}
The family $\Tcal$ is called \emph{shape-regular} if $\kappa_{\Tcal}$ is finite, and \emph{quasi-uniform} if $q_\Tcal$ is finite.
\end{definition}
Note that \Cref{def:regularuniform} applies equally to families with straight and curved elements. While $\kappa_{\Tcal}$ and $q_{\Tcal}$ represent the primary figures of merit for families of straight meshes, we require significantly more information.
We allow a mesh family $\Tcal = \{\Tcal_h\}_{h>0}$ to include curved mesh
elements with edges and possibly faces lying on a sufficiently piecewise-smooth boundary $\Gamma$, closely following~\cite{lenoir1986optimal,bernardi1989optimal}; we assume, as can be ensured by careful attention in constructing the straight mesh, that each element intersecting the boundary does not contain in the interior of a face a point of discontinuity. As is typically done (with the interesting exception of the recent work~\cite{ishizaka2026exact}), curved mappings from the reference $\hat{K}$ are constructed using the sum of an affine transformation, which gives the straight $\widetilde{K}$ that approximates the element, and a nonlinear perturbation to yield the final curved element $K$. In what follows, we summarize the construction of curved elements in Definitions~\ref{def:curved0} and~\ref{def:curved} and then present the regularity of the mappings in Theorem~\ref{thm:mapping_properties} and Corollary~\ref{cor:curvilinear}.\enlargethispage*{1ex}

\begin{definition}\label{def:curved0}
  A closed set $K \in \R^d$ is a \emph{curved $d$-simplex} if there exists a $C^1$ mapping $T_K: \hat{K} \to K$ of the form
  \begin{equation}\label{eq:Tk_decomp}
    T_K(\widehat{x}) \coloneqq \widetilde{F}_K(\widehat{x}) + \Phi_K(\widehat{x}), \qquad \widehat{x}\in\Khat
  \end{equation}
  where $\widetilde{F}_K(\widehat{x}) = \widetilde{B}_K \widehat{x} + \widetilde{b}_K$ is affine ($\widetilde{B}_K \in \Rbb^{3\times3}$ invertible satisfies~\cite{Ciarlet:72} the relation $\|\widetilde{B}_K\|_2 \le \frac{h_K}{\rho_{\KHat}}$) and $\Phi_K$ satisfies
  \begin{equation}\label{eq:mesh_reg_Ck_def}
    c_K \coloneqq \sup_{\widehat{x} \in \KHat} \|D\Phi_K(\widehat{x}) \cdot \widetilde{B}_K^{-1}\|_2<1.
  \end{equation}
  It follows that $T_K$ is invertible with $C^1$ inverse~\cite[Lem.\ 2.1]{bernardi1989optimal} and
  \begin{equation}\label{eq:curved_mesh_jac}
    \forall \widehat{x} \in \KHat,\quad (1 - c_K)^3 |\det \widetilde{B}_K| \le |\det DT_K(\widehat{x})| \le (1 + c_K)^3 |\det \widetilde{B}_K|.
  \end{equation}
We say that a family $\Tcal$ is \emph{regular} if there exists an $h_0>0$ such that
  \begin{equation}\label{eq:mesh_reg}
    c_1(\Tcal) \coloneqq \sup_{h\le h_0} \sup_{K \in \Tcal_h} c_K < 1.
  \end{equation}
A regular family $\Tcal$ is additionally \emph{regular of order $\theta$}, $\theta \in \N$, if (a) $F_K \in C^{\theta + 1}$ and (b) setting
  \[
    c_\ell(K) \coloneqq \sup_{\widehat{x} \in \KHat} \left\|D^\ell T_K(\widehat{x})\right\|_2 \cdot \lnrm \widetilde{B}_K \rnrm^{-\ell}_2, \; 2 \le \ell \le \theta + 1,\quad\mbox{for each}\quad K \in \Tcal_h \in \Tcal,
  \]
it holds that
  \begin{equation}\label{eq:curved_mesh_reg}
    c_\ell(\Tcal) \coloneqq \sup_{h \le h_0} \sup_{K \in \Tcal_h} c_\ell(K) < \infty, \quad 2 \le \ell \le \theta + 1.
  \end{equation}
\end{definition}

\begin{definition}[{\hspace{1sp}\cite{bernardi1989optimal}}]\label{def:curved}
  Let $\theta \in \N$. Let $\Gamma$ denote the
  boundary of a bounded region $\Omega \subset \R^3$, such that there exist a
  finite number of charts $\psi: \omega \to \R^3$, each with possibly different
  bounded domain $\omega \subset \R^2$, with $\psi \in C^{\theta
  + 1}(\bar{\omega})$ and together satisfying $\Gamma = \cup_{\psi}
  \{\psi(\xi), \xi \in \bar{\omega}\}$. Let $\{\tau_h\}_h$ denote a family of triangularizations of $\omega$, regular of order $\theta$, and let $\{\widetilde{\Tcal}_h\}_h$
  denote a family of tetrahedralizations of $\Omega$ with straight tetrahedra and
  assume that for each of the nodes $\alpha_i$ in $\tau_h$, $\psi(\alpha_i)$ is
  a vertex of $\widetilde{\Tcal}_h$; for each $\widetilde{K} \in \widetilde{\Tcal}_h$ which has two vertices on $\Gamma$, those vertices lay on the same chart; and  the set of all vertices $\widetilde{\Tcal}_h$ that lay on $\Gamma$ is the set $\cup_\psi \{\psi(\alpha): \alpha\mbox{ a vertex of } \kappa \in \tau_h\}$. If $\psi$ and $\psi'$ are two different charts with respective domains $\omega$ and $\omega'$, and $\{\tau_h\}$ and $\{\tau_h'\}_h$ denote families of exact triangulations of $\omega$ and $\omega'$, respectively, require that for, for each $h$, $\{\psi(\alpha), \alpha\mbox{ vertex of } \kappa \in \tau_h\} \cap \psi(\bar{\omega}) \cap \psi'(\bar{\omega}')$ and $\{\psi'(\alpha'), \alpha'\mbox{ vertex of } \kappa' \in \tau_h'\} \cap \psi(\bar{\omega}) \cap \psi'(\bar{\omega}')$ coincide.
  
  If at most one vertex of $K$ lays on $\Gamma$ then we set $\Phi_k \coloneqq 0$, hence $T_K \coloneqq \widetilde{F}_K$ and $K \coloneqq \widetilde{K}$, a straight simplex.
  
  Otherwise, $K$ has $j \in \{2, 3\}$ vertices on $\Gamma$. We then let $f_\kappa$ denote a (not necessarily unique) map $f_{\kappa}: \Delta^2 \to \kappa \subset \R^2$ which maps the reference
  $2$-simplex onto the pre-image of the set of vertices of a curved (face)
  triangle $\upsilon \subset \Gamma$ which contains all vertices in $K \cap
  \Gamma$ and which satisfies $f_\kappa(\hat{\alpha}_i) = \alpha_i$ for $i=1,\ldots,3$.
  Let $\theta \in \N$ and let $\pi_{\kappa}^l$, $1 \le l \le \theta$, denote~\cite[\S 6]{bernardi1989optimal} the classical
  $l$\textsuperscript{th}-order interpolation operator on $\kappa$ with data at the $\binom{l+2}{l}$ points on the
  so-called principal lattice of $\Delta^2$~\cite{nicolaides1972class}. The nonlinear deviation from affine $\Phi_K$ in~\eqref{eq:Tk_decomp} is then set to
\begin{multline}
    \Phi_K(x_1, x_2, x_3) \\
 \coloneqq \Big\{ (x_1 + \cdots +  x_j)^{\theta + 2} (\psi - \pi_{\kappa}^\theta \psi)
    + \sum_{l=2}^\theta (x_1 + \cdots + x_j)^l(\pi_{\kappa}^l\psi - \pi_{\kappa}^{l-1} \psi) \Big\}\\
    \circ f_{\kappa}\big(\,(x_1 \hat{\alpha}_1 + \cdots + x_j \hat{\alpha}_j)/(x_1 + \cdots + x_j)\,\big),
\end{multline}
with the meaning of $\pi_\kappa^\theta \psi \circ f_\kappa(x)$ for $x \in \Delta^2$ as in~\cite[\S 2]{lenoir1986optimal}.
\end{definition}
\begin{remark}\label{rem:curved}
  In treating the arbitrary-dimensional case,
  reference~\cite{bernardi1989optimal} reasons from recursion in dimension that
  one always has available an `exact' mesh $\tau_h$ of a \emph{possibly-curved}
  chart domain $\omega$, and for every element in $\tau_h$ a map $f_\kappa$ that is regular of order $\theta$ (in the sense of \Cref{thm:mapping_properties}). Such a mesh may be nontrivial to construct from first principles; our implementation assumes a polygonal domain $\omega$ for each chart
  $\psi$, and is considerably simpler. However, and actually most
  practical in a setting where mesh generation software will readily produce a
  straight tetrahedralization $\widetilde{\mathcal{T}}_h$ of $\Omega$, we use the vertices in $\widetilde{\mathcal{T}}_h \cap \Gamma$ to define, via their pre-image,
  a straight triangularization of a chart mesh $\tau_h$---the resulting mesh then
  implicitly defines a polygonal domain $\omega$. As a result, the map $f_\kappa$ satisfies $f_\kappa: \Delta^2 \to \psi^{-1}(\upsilon)$. With this approach, if multiple charts are required then one of the domains may contain elements with non-affine maps $f_\kappa$, i.e. maps to the domain of the chart that are curvilinear reflecting curved mesh elements in the mesh $\tau_h$ of the chart domain; the implementation in \texttt{Inti.jl} does not explicitly treat this and this extension is left for future work, with all numerical examples on curved domains in this paper therefore requiring only a single chart.
\end{remark}

The resulting curved simplices possess the following regularity properties, which will be useful
in recovering high-order estimates.
\begin{theorem}[{\hspace{1sp}\cite[Thm.\ 6.2, Corr.\ 6.2]{bernardi1989optimal}}]\label{thm:mapping_properties}
  Let $\theta \in \N$ be given and suppose the boundary $\Gamma$ is piecewise-$C^{\theta+1}$. Then $\Phi_K: \KHat \to \R^3$ is $C^{\theta + 1}$ and satisfies
  \begin{equation}
    \sup_{\widehat{x} \in \KHat} \left\|D^\alpha \Phi_K(\widehat{x})\right\|_2 \lesssim h_{\kappa}^{\max(2, |\alpha|)}, \quad 1 \le |\alpha| \le \theta + 1,
  \end{equation}
  where $h_\kappa$ denotes the diameter of the mesh element in the chart domain $\omega$.
  Moreover, there exists $h_0>0$ such that the resulting mesh family $\{\Tcal_h\}_{h \le h_0}$ is regular of order $\theta$.
\end{theorem}

We collect some facts that immediately follow which will be useful in the sequel.

\begin{corollary}\label{cor:curvilinear}
    Let $\theta \shin \N$ be given and let $\Gamma$ be piecewise-$C^{\theta+1}$. There exists $h_0 \shg 0$ such that $\Tcal = \{\Tcal_h\}_{h\le h_0}$ is regular of order $\theta$ and such that for each element $K \in \Tcal_h \in \Tcal$ with map $T_K$ provided by~\cref{eq:Tk_decomp}, each component $[T_K]_i$, $i \in \{1, 2, 3\}$ of $T_K$ satisfies
    \begin{equation}\label{eq:Dalpha_TK}
      \sup_{\widehat{x} \in \KHat} \left|D^\alpha [T_K(\widehat{x})]_i\right| \le C_1 h^{|\alpha|}, \quad |\alpha| \le \theta + 1,\, C_1 = C_1(\Gamma, \kappa_\Tcal, c_{|\alpha|}(\Tcal), |\alpha|),
    \end{equation}
    and the Jacobian map satisfies
    \begin{equation}\label{eq:Jac_TK}
        \forall \widehat{x} \in \KHat,\quad \quad C_2 h^3 \le \left|\det T_K(\widehat{x})\right| \le C_3 h^3, \quad C_j = C_j(\kappa_\Tcal, c_1(\Tcal)) > 0.
    \end{equation}
    Moreover, using a mesh family $\mathcal{T}$ that is regular of order $\theta$ and a quadrature rule of degree of exactness $m$ on reference elements results in an order of convergence of $\min(m+1, \theta+1)$ for sufficiently smooth integrands.
\end{corollary}

\subsection{Fast summation-compatible approximation formula} For each $K\in\Tcal_h$, the density interpolant $f_n=f_n(\cdot;K)$ of $f$ over $K$ is, in principle, given by
\begin{equation}
  f_n(y:K) = \sum_{i=1}^{q_n} f(x_i)\lambda_i(y;K), \label{f_n:lagr}
\end{equation}
in terms of the Lagrange polynomials $\lambda_i(y;K)$ associated with the interpolation node set $\Ical_n\sheq\lcb x_i: x_i\shin K \rcb_{i=1}^{q_n}$. The 3D-VR quadrature nodes of $K$ may be (and, in this work, are) used as interpolation nodes, setting $x_i\sheq\xi_i,\ i\shin\rmQ(K)$ (see~\eqref{eq:standard_QR}). A significant difficulty in using local interpolants such as~\eqref{f_n:lagr} is that the volume integral yielding $\ZcalB[f_n](x)$ in~\eqref{Z:punct} and the boundary integrals yielding $\Bcal_{\Zcal}[f_n,\Xi_n](x)$ in~\eqref{Z:reg} both depend on the target point $x\shin K$ explicitly through the Green's function, but also implicitly through $f_n(\cdot; K)$. Clearly, evaluating the volume integral $\Vcal[f - f_n(\cdot; K)](x)$ over $\Omega$ at all target points $x\shin K$ for every $K\shin\Tcal_h$ by means of the quadrature rule~\eqref{eq:standard_QR},  would result in a highly inefficient algorithm even using fast methods. Similar issues would affect the layer potentials in~\eqref{V:IPP}, \eqref{W:IPP2} and~\eqref{X:IPP}.

To remedy this potential bottleneck, we use the binomial theorem to express every \emph{local} interpolant $f_n(\cdot; K)$ in a \emph{fixed, global} basis of normalized monomials, i.e.,
 \begin{equation}\label{eq:mon_interp}
     f_n(y; K)=\sum_{|\alpha|\leq n}c_\alpha[f](K)p_\alpha(y),\qquad  p_\alpha(y) \coloneqq \frac{y^\alpha}{\alpha!},
\end{equation}
so that the local character of $f_n$ is confined to the coefficients $c_{\alpha}[f](K)$. In~\eqref{eq:mon_interp} and hereafter, we make use of the standard multi-index notation where, for any $\alpha = (\alpha_1,\alpha_2,\alpha_3)\in\N_0^3$, we set $\alpha! = \alpha_1!\alpha_2!\alpha_3!$, $|\alpha| = \alpha_1+\alpha_2+\alpha_3$, $y^\alpha=y_1^{\alpha_1} y_2^{\alpha_2} y_3^{\alpha_3}$ when $y=(y_1,y_2,y_3)\in\R^3$ and, for $\beta=(\beta_1,\beta_2,\beta_3)\in\N_0^3$ with $\beta\leq\alpha$ (i.e. $\beta_i\leq\alpha_i$),
\[
{\binom{\alpha}{\beta}}
 =\frac{\alpha!}{(\alpha-\beta)!\beta!}.
\]
For given $f$ and $K$, determining the coefficients $c_{\alpha}[f](K)$ entails solving linear systems with $q_n$ unknowns enforcing the interpolation conditions for each scalar component of $f$, at a $\Ocal(q_n^3)$ $f$-independent computational cost for setting up and factoring the dense matrix followed by a $\Ocal(q_n^2)$ $f$-dependent cost for back-substitution.\enlargethispage*{7ex}

Using~\eqref{eq:mon_interp} and for given $\OO$, we hence evaluate the term $\ZcalB[f_n(\cdot;K)](x)$ of the puncturing approximation~\eqref{Z:punct} as
\begin{equation}
  \ZcalB[f_n(\cdot;K)](x)
 \approx \sum_{|\alpha|\leq n}c_\alpha[f](K) \, \Qcal^{h,m}_{\Omega\setminus K}\lsqb \sfKBr(x,\dotp) p_\alpha \rsqb \label{Z:punct:sum}
\end{equation}
(applying the approximate quadrature~\eqref{eq:standard_QR}), with $\Qcal^{h,m}_{\OO\setminus K}[\sfKBr(x,\dotp) p_\alpha](x)$ precomputed once for each $p_{\alpha}$ and then re-used for each new source $f$, merely with updated weights $\{c_\alpha[f](K)\}_{|\alpha|\leq n}$. The regularizing term $\Zcal[f_n](x)$ of~\eqref{Z:split} is written likewise as
\begin{equation}
  \Zcal[f_n(\cdot;K)](x)
 = \sum_{|\alpha|\leq n}c_\alpha[f](K) \, \lcb  \mu(x)\Xi_{\alpha}(x) + \Bcal_{\Zcal}[p_{\alpha},\Xi_{\alpha}](x) \rcb \label{Z:reg:sum}
\end{equation}
with the generic notations of~\eqref{Z:reg}. Finally, using~\eqref{Z:punct} in~\eqref{Z:split} with $\ZcalB[f_n(\cdot;K)](x)$ and $\Zcal[f_n(\cdot;K)](x)$ respectively evaluated as indicated by~\eqref{Z:punct:sum} and~\eqref{Z:reg:sum}, the generic volume potential $\Zcal[f](x)$ is approximated as, denoting $\Bcal_{\Zcal}^h$ an approximation of $\Bcal_{\Zcal}$,
\begin{align}\label{Z:eval}
\Zcal[f](x)
 & \approx \ZcalB[f](x) - \ZcalB[f_n(\cdot;K)](x) + \Zcal\lsqb f_n(\cdot;K) \rsqb(x) \\
 &\approx \Qcal^{h,m}_\Omega[\sfKBr(x,\cdot)f] \notag \\ & \quad  + \sum_{|\alpha|\leq n}c_\alpha[f](K) \, \Lcb
    - \Qcal^{h,m}_{\OO\setminus K}[p_\alpha](x) + \mu(x)\Xi_{\alpha}(x) + \Bcal_{\Zcal}^h[p_{\alpha},\Xi_{\alpha}](x) \Rcb. \notag
\end{align}

\subsection*{Complexity estimation}
The proposed approximation formula~\eqref{Z:eval} leads to an algorithm with offline ($f$-independent) and online ($f$-dependent) components. The former can be accelerated with fast algorithms, while the $f$-dependent weights are \emph{efficiently computable} in the online phase. Its efficiency may be assessed by estimating the computational complexity of evaluating $\Zcal[f]$ at all the $N$-numbered volume quadrature nodes $\{\xi_j\}_{j=1}^N$. Table~\ref{tab:complexity-estimates} provides an overview of the computational costs associated with these evaluations, wherein the costs for evaluating $\Qcal_\Omega[\sfKBr(x,\cdot)f]$, $\Qcal_\Omega[\sfKBr(x,\cdot)p_\alpha]$ and $\Bcal_{\Zcal}^h[p_{\alpha},\Xi_{\alpha}](x)$ assume FMM-accelerated summations at the $N$ volume target points. The latter two tasks can be effected in embarrassingly-parallel fashion across the polynomial degree multi-index~$\alpha$. The FMM-accelerated general-purpose DIM~\cite{faria2021general} is utilized in this work to evaluate the boundary integrals $\Bcal_{\Zcal}$, with the boundary discretized using $N_b\propto N^{2/3}$ points. Evaluating the coefficients $c_\alpha[f](K)$, $K\in\Tcal_h$ has a cost of $\Ocal(q_n^3 L) = \Ocal(q_n^2 N)$. Setting up the interpolation system matrices for every mesh element costs $\Ocal(q_n^2 N)$, while the subsequent evaluation of $c_\alpha[f](K)$, $K\in\Tcal_h$, given $f$
costs $\Ocal(q_n N)$. Overall, the operation count estimate shows that, for a given interpolation degree $n\in\N_0$, the VDIM methodology achieves quasi-optimal complexity.

As $q_n=\Ocal(n^3)$, the computational costs however grow quickly with the chosen interpolation degree $n$, making our method best suited for moderate values of $n$; additionally, stability concerns discussed in~\cite[\S 8]{anderson2024fast} also limits the direct practical use of this method at arbitrarily high orders.\enlargethispage*{1ex}

\begin{table}[t]
\begin{center}
\begin{tabular}{cllc}
\toprule
 & \multicolumn{2}{c}{Task} & Cost \\\cmidrule{2-4}
  \multirow{2}{*}{$f$-dependent} & $\Qcal^{h,m}_\Omega[\sfKBr(\xi_j,\cdot) f]$, & $j=1,\ldots,N$ &
  $\Ocal(N{\log N})$  \\
  & $c_\alpha[f](K_\ell)$, & $\ell=1,\ldots,L$, $|\alpha|\leq n$ & $\Ocal(q_n N)$ \\
\cmidrule{2-2}\cmidrule{3-3}\cmidrule{4-4}
  \multirow{4}{*}{$f$-independent} & $\Qcal^{h,m}_\Omega[\sfKBr(\xi_j,\cdot)p_\alpha(\cdot)],$
  & $j=1,\ldots,N,$ $|\alpha|\leq n$ & $\Ocal(q_n N {\log N})$ \\
  & $\Bcal_{\Zcal}^h[p_{\alpha},\Xi_{\alpha}](\xi_j)$, & $j=1,\ldots,N,$ $|\alpha|\leq n$
  & $\Ocal(q_nN {\log N})$    \\
  & $\mu(\xi_j)P_\alpha(\xi_j)$, & $j=1,\ldots,N,$ $|\alpha|\leq n$ & $\Ocal(q_nN)$ \\
  & interpolation system, & $\ell=1,\ldots,L$ (setup) & $\Ocal(q_n^2N)$\\\bottomrule
\end{tabular}
\end{center}
  \caption{Computational costs for the online ($f$-dependent) and offline ($f$-independent) components of the method (both are required). Here $N = L q_n$ is the number of volume evaluation points in a mesh with $L$ elements, with the cardinality $q_n$ of the polynomial basis given by~\eqref{qn:expr}.}\label{tab:complexity-estimates}
\end{table}

\section{Error estimates for the global regularization}
\label{sec:analysis}

This section provides error estimates under mesh refinement for the puncturing-approximation component of the proposed VIO evaluation method~\eqref{Z:eval}, stated and discussed for convenience in terms of the generic VIO $\Zcal$. The analysis to follow applies to any volume potential whose kernel verifies the following admissibility condition:
\begin{definition}\label{def:kernel_bounds}
  A kernel function $\sfK$, possibly tensor-valued, is \emph{admissible of class
  $\sigma$} if, for each $R > 0$ there exists a positive integer $\sigma \leq 3$ and a constant $C_{\sfK,|\eta|,R}>0$ such that for each multiindex $\eta\in\Nbb_0^3$ and for every $x \in \R^3$ satisfies
  \begin{equation}\label{eq:greens_admissible}
    \left\|D_{y}^{\eta} \sfK(x,y)\right\| \le C_{\sfK,|\eta|,R} |x - y|^{-\sigma - |\eta|}, \quad y \in B_R(x)\setminus\{x\}.
  \end{equation}
If $\sigma\sheq3$, the Cauchy principal value in~\eqref{X:expr} is additionally required to be finite.
\end{definition}
Here, and in the sequel, $D_y^\eta \sfK$ denotes element-wise partial derivatives of the kernel $\sfK$, and $\|\cdot\|$ refers to the vector/matrix/tensor element-wise maximum norm; since the VIO concerns a tensor contraction over all indices in the lower-order tensor argument to $\Zcal$, constants in the estimates that we establish will generally depend on $p$ and $d$. The cases $\sigma=1$, $\sigma=2$ and $\sigma=3$ respectively correspond to the VIOs $\Vcal$, $\Wcal$ and $\Xcal$ introduced in Section~\ref{sec:prelim}, the latter involving a strongly singular integral.

The forthcoming convergence result (\Cref{thm:tri_error_analysis}) applies for evaluation at sets of points of $\OmegaB$ that lay in the interior of an element with respect to other
elements of the mesh, as made precise in the following definition:
\begin{definition}\label{def:wellseparatedeval}
  Consider a mesh family $\Tcal=\{\Tcal_h\}_{h>0}$, and let $\Scal_h=\bigcup_{K\in\Tcal_h}\p K$. A  family of evaluation point sets $\Ecal=\{\Ecal_h\}_{h>0}$, $\Ecal_h\subset\overline{\Omega}$ for all $h>0$, possibly intersecting the boundary $\G=\p\Omega$, is called \emph{well-separated} with respect to $\Tcal$ if
\begin{equation}\label{eq:wellseparated_dist}
    d_{\Tcal,\Ecal} \coloneqq \inf_{h>0}\,\,\inf_{\substack{y \in \Scal_h\setminus\G\\ x \in \Ecal_h}} \frac{|x - y|}{h} > 0.
\end{equation}
\end{definition}

A typical example is evaluation at the point set consisting of 3D-VR interpolation/quadrature nodes, which are interior to elements: the theorem thus in particular estimates the error in the VIO evaluated at all the 3D-VR nodes, i.e. for $\Ecal_h=\lcb \xi_j \rcb_{j=1}^N$ with $N = Lq_n$. Indeed, any set of interior points on $\KHat$ leads to a well-separated family of evaluation points over the triangulation, as the following proposition (stated for a straight mesh family but extending easily to a family $\Tcal$ with curved elements) shows. Another use-case of the definition would be the evaluation at point-sets lying on the boundary $\G=\partial \Omega$, e.g.\ for the solution of boundary value problems.\enlargethispage*{3ex}

\begin{proposition}[{\hspace{1sp}\cite[Prop.\ 6.3]{anderson2024fast}}]\label{rem:interior_VR_nodes}
    Let $\Tcal=\{\Tcal_h\}_{h>0}$ be a quasi-uniform and
    shape-regular family of straight meshes of $\Omega$. The family $\Ecal=\{\Ecal_h\}_{h>0}$ with $\Ecal_h$ being a set of interior
    quadrature nodes over a triangulation $\Tcal_h\subset\Tcal$ of
    $\Omega$, is a well-separated family of evaluation point sets with respect
    to $\Tcal$.
\end{proposition}

The convergence analysis finally also relies on the existence and boundedness of Lagrange polynomials on mesh elements in the family:
\begin{assump}\label{assump:lagrange}
For a given mesh family $\Tcal=\{\Tcal_h\}_{h>0}$ of $\Omega$, the
  Lagrange interpolation polynomials $\lambda_i(y) = \lambda_i(y;
  K_\ell)$ ($1 \le i \le q_n$) exist on each element $K_\ell\in\Tcal_h\in\Tcal$ and, further, are uniformly bounded; that is,
  \begin{equation}\label{eq:lagrange_assumption_bound}
    \sup_{\Tcal_h \in \Tcal} \max_{K_\ell \in \Tcal_h} \max_{y \in H(K_\ell)} \max_{1 \le i \le q_n} |\lambda_i(y; K_\ell)| \le \Lambda(n,\Tcal),
  \end{equation}
  with $H(K_\ell)$ being  the convex hull of the vertices and the set of  interpolation nodes $ \Ical_n$.
\end{assump}
In the case of straight elements, this is a mild element-quality assumption since, per \cite[Thm.\
A.2]{anderson2024fast}, bounds on $\Lambda(n, \Tcal)$ can be expressed in terms of $n$
and the shape-regularity constant $\kappa_{\Tcal}$.

\subsection{Main convergence result}

We can now state the main convergence result:
\begin{theorem}\label{thm:tri_error_analysis}
  Let $n\in \N_0$ be the interpolation degree, let the (positive weight) quadrature have degree of exactness $m\in \N_0$, $m > n$, and let $\Gamma=\partial\Omega$ be piecewise-$C^{m+2}$. Let $\Tcal = \left\{\Tcal_h\right\}_{0<h\le h_0}$ denote a mesh family of the connected domain~$\Omega$  that is shape-regular and quasi-uniform, satisfies \Cref{assump:lagrange}, and has maximum element size $h_0$ chosen so that \Cref{cor:curvilinear} holds with $\theta \ge m + 1$. Let $\Ecal=\{\Ecal_h\}_{h>0}$ be a family of well-separated evaluation point sets for $\Tcal$, and denote by $f_n(\cdot; K)$ the Lagrange polynomial interpolant of $f\in C^{s}(\overline\Omega)$, $s = \max\{n+3,m+1\}$, with interpolation in some $K\in\Tcal_h$ enforced on the associated interpolation node-set $\Ical_n\subset K$. Finally, let $\Zcal$ denote a volume integral operator~\cref{VIO:Z} with kernel $\sfK$ that is admissible of class $\sigma$ (\Cref{def:kernel_bounds}). For all $x \in  K\cap\Ecal_h$ it holds that
\begin{multline}\label{eq:triangle_error_estimate_vr}
      \left\|\Zcal\left[f - f_n(\cdot; K)\right](x) - \ZcalB^{h,m}\left[f - f_n(\cdot; K)\right](x)\right\| \\
    \le C^{(1)}_{\Zcal} h^{n+4-\sigma} + C^{(2)}_{\Zcal}
        \begin{cases}
            h^{n+4-\sigma}, & m > n + 3 - \sigma \\[-0.5ex]
            h^{n+4-\sigma}\left|\log h\right|, &m = n + 3 - \sigma\\[-0.5ex]
            h^{m+1}, & m < n + 3 - \sigma,
        \end{cases}
\end{multline}
where $C^{(1)}_{\Zcal}$ and $C^{(2)}_{\Zcal}$ are positive constants independent of $h$ but dependent on $p$, $d(\sheq3)$, $\Omega$, $\kappa_\Tcal$, $q_\Tcal$, $f$, $m$, $n$, $d_{\Tcal, \Ecal}$ and $\Lambda(n, \Tcal)$; the constants depend linearly on $\Lambda(n, \Tcal)$.
\end{theorem}

\begin{remark}
    The applicability of \Cref{thm:tri_error_analysis} to families of curved tetrahedra is a strengthening of the analogous 2D result in \cite{anderson2024fast} which was limited to mesh families consisting of straight triangles. But in view of the curved mapping used in~\cite{anderson2024fast} also satisfying \Cref{cor:curvilinear}, per the analysis in~\cite{zlamal1973curved}, the same proof technique used for \Cref{thm:tri_error_analysis} extends~\cite[Thm.\ 6.6]{anderson2024fast} to mesh families with curved triangles.
\end{remark}

\begin{proof}
Fix an element $K\in\Tcal_h$, and let $x \in K\cap\Ecal_h$ be a fixed evaluation point.
Let $B_r(x):=\lcb y\in\R^3:\, |x-y|\leq r \rcb$ be the closed ball centered at $x$ with radius $r$. We let $\Ncal_h$ be the neighborhood of $x$ defined by
\begin{equation}
  \Ncal_h = \bigcup\ \lcb K_{\ell}\in\Tcal_h:\, K_{\ell}\cap B_h(x)\not=\emptyset \rcb, \label{Nh:def}
\end{equation}
noting that $\Ncal_h\subset B_{2h}(x)$.
The proof is then divided into in two main parts. The first part proceeds by showing
that the near-singular volume integrals over each of
  the elements within $\Ncal_h$ are each
  $\Ocal(h^{n+4-\sigma})$ quantities, in particular showing that neglecting
  the integral over $\Ncal_h \setminus K$  approximates the integral over $\Ncal_h$
  with errors that behave as $\Ocal(h^{n+4-\sigma})$ as $h
  \to 0$. The second part considers the error from numerical quadrature
  on the elements in the outer region $\Omega\setminus\Ncal_h $.

\proofstep{Part 1: contributions from $\Ncal_h$.}
In what follows the Lagrange interpolation error estimate~\cite[Thm.\ 2]{Ciarlet:72} will prove useful; it provides
\begin{equation}\label{f_fn_interp_near_estimate}
  \left\|f(y) - f_n(y; K)\right\| \le {C}_0 h^{n+1}\quad\mbox{for all}\quad y \in \Ncal_h(x),
\end{equation}
where ${C}_0 = C_0(f)$ denotes a constant independent of $h$.  Note that these estimates hold not merely on $K$, with $\Ical_n \subset K$, but as well over an $\Ocal(h)$ neighborhood.

The first step of the proof provides bounds on components of the regularized volume potential $\Zcal[f - f_n(\cdot; K)](x)$, for whose integrand the shorthand notation
\begin{equation}\label{eq:phi1_def}
    \phi(x,y) \coloneqq
    \sfK(x,y) [f(y) - f_n(y; K)], \quad\quad x, y \in \widebar{\Omega},
\end{equation}
is used thereafter. For the weakly singular VIOs $\Zcal=\Vcal,\Wcal$, for which $\sigma=1,2$, considering each of the integrals over the elements of $\Ncal_h$, we have from a change to spherical coordinates centered at  $x\in K\cap\Ecal_h\subset\Ncal_h$ and the bound~\cref{f_fn_interp_near_estimate}, the estimate
\begin{equation}\label{eq:int_near_sing}
  \bigg\|\, \int_{\Ncal_h} \phi(x,y) \de y \,\bigg\|
 \leq C_{p,d}\int_{\Ncal_h} \left\| \phi(x,y) \right\| \de y \le {C}_1 h^{n+4-\sigma},
\end{equation}
which in view of the fact that $K\subset\Ncal_h(x)$ implies that
\begin{equation}\label{Vk_analytical_regularization_error}
  \lnrm \Zcal\left[f - f_n(\cdot; K)\right](x) - \ZcalB\left[f - f_n(\cdot; K)\right](x) \rnrm \le {C}_1 h^{n+4-\sigma}.
\end{equation}
Similarly, since $\Ncal_h\setminus K\subset\Ncal_h$ we get from the last inequality of~\cref{eq:int_near_sing} the bound
\begin{equation}\label{Gk_near_estimate_analytical}
  \bigg\|\, \int_{\Ncal_h\setminus K} \phi(x,y) \de y \,\bigg\| \le {C}_1 h^{n+4-\sigma}.
\end{equation}

Turning to the singular potential $\Zcal=\Xcal$, the contribution to $\Xcal[f - f_n(\cdot; K)]$ from $\Ncal_h$ is given by~\eqref{X:expr} with $g=f - f_n(\cdot; K)$ and $\Ncal_h$ replacing $\OO$, and in particular includes the (local at $x\shin K$) term $S\cdot g(x)$. From writing $g(y)=[g(y)-g(x)]+g(x)$, applying Green's identity to the volume integral over $\Ncal_h\setminus B_{\eps}(x)$ with density $g(x)$, cancelling the resulting integral on the sphere $\partial B_{\eps}(x)$ with that defining $S$ in~\eqref{S:def}, and finally taking the limit $\eps\to0$ that yields the principal value in~\eqref{X:expr},  we obtain
\begin{multline}\label{aux01}
  \Xcal[\lpar f - f_n(\cdot; K) \rpar\mathds{1}_{\Ncal_h}] =
  - \int_{\Ncal_h} \phi^{\Xcal}(x,y) \de y \\ - \int_{\partial\Ncal_h} \nu(y)\,\lpar \nabla_x\GG(x,y) \cdot[f - f_n(\cdot; K)](x) \rpar \ds(y),
\end{multline}
with $\phi^{\Xcal}(x,y):=\nabla_x\nabla_y\GG(x,y)\cdot\lpar [f(y) - f_n(y; K)]-[f(x) - f_n(x; K)]\rpar$. Using the estimate~\eqref{f_fn_interp_near_estimate} in the zeroth-order Taylor expansion of $f(y) - f_n(y; K)$ about $y=x$ with integral remainder, we find
\[
  \left\| \lpar f(y) - f_n(y; K) \rpar  -\lpar f(x) - f_n(x; K) \rpar \right\| \leq C_2|y-x| h^n,
\]
so that the integral over $\Ncal_h$ in~\eqref{aux01} satisfies the bound~\eqref{eq:int_near_sing}, and that over $\Ncal_h\setminus K$ the bound~\eqref{Gk_near_estimate_analytical}. Then, e.g. using again spherical coordinates, we have
\[
   \left\| \int_{\partial\Ncal_h} \nu(y)\,\lpar \nabla_x\GG(x,y) \cdot[f - f_n(\cdot; K)](x) \rpar \ds(y) \right\| \leq C_3h^0
\]
(and similarly with the above integral over $\partial K$) which, combined with~\eqref{f_fn_interp_near_estimate} for $y=x$, produces bounds of the form $Ch^{n+1}$ for either boundary integral. Summarizing, since $\sigma=3$ in this case, we again obtain bounds of the form~\eqref{eq:int_near_sing} or~\eqref{Gk_near_estimate_analytical}
\begin{equation}
\begin{aligned}
  \lnrm\Xcal\lsqb \lpar f - f_n(\cdot; K)\rpar \mathds{1}_{\Ncal_h}(\cdot)\rsqb(x)\rnrm &\le C_4 h^{n+4-\sigma},\\
  \lnrm\Xcal\left[f - f_n(\cdot; K)\right](x) - \widebreve{\Xcal}\left[f - f_n(\cdot; K)\right](x)\rnrm &\le C_4 h^{n+4-\sigma}.
\end{aligned}\label{Xbreve_regularized}
\end{equation}

Turning to numerical quadratures and using the rule~\eqref{eq:standard_QR}, for the elements comprising $\Ncal_h\setminus K$, using~\cref{Gk_near_estimate_analytical} and~\cref{f_fn_interp_near_estimate} in conjunction with the triangle inequality, we find, contracting $\sfK$ on $\phi$ and picking up a factor $C_{p,d}$,
\begin{align}\label{near_quadrature_estimate}
\MoveEqLeft[6]
    \bigg\| \,\int_{\Ncal_h \setminus K} \phi(x,y) \de y
    - \sum_{j\in\rmQ(\Ncal_h \setminus K)} \upomega_{j} \phi(x,\xi_j) \,\bigg\| \\
    & \le ({C}_1+C_4)h^{n+4-\sigma} + {C}_0 C_{p,d} h^{n+1} \max_{j\in\rmQ(\Ncal_h \setminus K)} \left\|\sfK(x, \xi_j)\right\| \sum_{j\in\rmQ(\Ncal_h \setminus K)} \upomega_{j} \notag \\
    & \le C_5 h^{n+4-\sigma}, \notag
\end{align}
where the inequalities above follow because, firstly, the quadrature rule has positive weights that satisfy $\sum_{j\in\rmQ(\Ncal_h \setminus K)}\upomega_{j} \lesssim h^3$ and secondly, since $x \in \Ecal_h$ lies at a distance from $\Ncal_h\setminus K$ that scales linearly with $h$, from~\cref{eq:greens_admissible} we have $\sup_{j\in\rmQ(\Ncal_h \setminus K)} \left\|\sfK(x, \vv\xi_j)\right\| \lesssim h^{-\sigma}$, with an implied constant dependent on $d_{\Tcal,{\Ecal}}$ (see \Cref{def:wellseparatedeval}) and $C_{\sfK, 0, \operatorname{diam} \Omega}$. We emphasize that all constants in Part 1 above depend on $p$ and $d(\sheq3)$ due to tensor contraction and the non-sub-multiplicativity of $\|\cdot\|$.

\proofstep{Part 2: contributions from $\Omega\setminus\Ncal_h$.}
This part relies on a result concerning the accuracy of ordinary quadratures over  $\Omega \setminus \Ncal_h$, given thereafter in \Cref{lem:ordquad_convergence_farfield_optimal}. That result in turn requires sharp estimates on the derivatives of $\phi(x,y)$, given in the following lemma whose proof is based on error estimates for multivariate interpolation. In what follows, in Lemmas~\ref{taylor_lagrange_interpolation_lemma}\textendash\ref{lem:ordquad_convergence_farfield_optimal}, we let $f, f_n$, and $K$ denote, without loss of generality, arbitrary scalar components of the respective tensor quantities; in the proof completion, the component-wise bounds combine to yield the final proof of \Cref{thm:tri_error_analysis} for the full tensor contraction. In the sequel we define $S_n(t):=\sum_{\gamma:|\gamma|\leq n} t^{n+1-|\gamma|}$.

\begin{lemma}\label{taylor_lagrange_interpolation_lemma}
Take as given the assumptions and setting of \Cref{thm:tri_error_analysis}. In particular, $n\in\N_0$ is the interpolation degree and $m\in\N$. Letting  $\alpha\in\N_0^3$ be a multi-index verifying $|\alpha| \le m+1$ and for $\phi$ given in~\eqref{eq:phi1_def}, the following estimate holds for
any $\delta>0$ and any $x \in K$:
\begin{equation}\label{eq:interpolation_estimate}
    \left|D_{y}^{\alpha} \phi(x,y)\right| \le
    \delta^{-\sigma + n+1-|\alpha|}\Lpar C\subT+  C\subL \Lambda S_n(h/\delta)\Rpar,\quad y \in\Omega\setminus B_\delta(x),
\end{equation}
where the positive constants $C\subT$ and $C\subL$ satisfy $C\subT = C\subT(|\alpha|,n, f,\Omega, C_{\sfK,|\alpha|,\operatorname{diam}\Omega}) $ and $C\subL = C\subL(|\alpha|,n, \kappa_\Tcal, q_\Tcal, f, \Omega, C_{\sfK,|\alpha|,\operatorname{diam}\Omega})$, but are both independent of $\Lambda$, $x$, $\delta$, $K$, $\Tcal$ and $h$.
\end{lemma}

We will develop high-order composite quadrature estimates for collections of curvilinear tetrahedral elements by leveraging the regularization result \Cref{taylor_lagrange_interpolation_lemma} and extending it to hold for the pull-back to the reference element $\KHat$. Thus, bounds on the derivatives of the composition of $\phi$ with the curvilinear mapping $T_\ell$ are obtained via a multivariate version of Fa\`a di Bruno's formula. A 2D analogue of this lemma extends the high-order estimates in~\cite{anderson2024fast} to collections of curvilinear elements, which were not previously treated.
\begin{lemma}\label{taylor_lagrange_interpolation_lemma_reference}
  Take as given the assumptions and setting of \Cref{thm:tri_error_analysis}, so that in particular $\Tcal_h \in (\Tcal)_{0<h\le h_0}$, $\Tcal_h$ a mesh with $L$ elements and $h_0$ a value such that \Cref{cor:curvilinear} holds with $\theta \ge m + 1$ for all $0 < h \le h_0$. Let $\delta>0$ and $x \in K \cap \mathcal{E}_h$ be given for a certain $K \in \Tcal_h$, and let $\ell \in \N$, $1 \le \ell \le L$, be such that the element $K_{\ell}\in\Tcal_h$ satisfies $K_{\ell}\subset \OO\shsetm \overline{B_{\delta}(x)}$.  Then for all multi-indices $\beta$, $|\beta| \le m+1$, it holds that for all $\widehat y \in \KHat$,
\begin{equation}\label{eq:interpolation_estimate_reference}
  \left| D_{\widehat{y}}^\beta \phi(x, T_\ell(\widehat{y}))\right|
  \le h^{|\beta|}\sum_{\lambda: 1 \le |\lambda| \le |\beta|} \delta^{-\sigma + n + 1 - |\lambda|}\Big( \widehat C\subT + \widehat C\subL \Lambda S_n(h/\delta) \Big),
\end{equation}
  with positive constants $\widehat C\subT$ and $\widehat C\subL$ satisfying $\widehat C\subT = \widehat C\subT(n, |\beta|, c_{|\beta|}(\Tcal), f,\Omega, C_{\sfK,|\beta|,\operatorname{diam}\Omega}) $ and $\widehat C\subL = \widehat C\subL(n, |\beta|, \kappa_\Tcal, q_\Tcal, c_{|\beta|}(\Tcal), f, \Omega, C_{\sfK, |\beta|,\operatorname{diam}\Omega})$, but both independent of $\Lambda$, $x$, $\delta$, $K$, $\Tcal$ and $h$.
\end{lemma}
\begin{proof}
  Calling $y = T_\ell(\widehat y)$, the multivariate Fa\`a di Bruno formula~\cite{constantine1996multivariate} provides
  \begin{equation}\label{eq:faa_di_bruno}
    D_{\widehat y}^\nu \phi(x, T_\ell(\widehat y)) = \sum_{\lambda: 1 \le |\lambda| \le |\nu|} D_y^\lambda \phi(x, y)\sum_{s=1}^{|\nu|} \sum_{p_s(\nu, \lambda)} (\nu!) \prod_{j=1}^s \frac{[D^{l_j}_{\widehat y} T_\ell]^{k_j}}{(k_j!)[l_j!]^{|k_j|}}
  \end{equation}
where
  \begin{equation}\label{eq:psdef}
  \begin{split}
    p_s(\nu, \lambda) = \{(k_1, \ldots, k_s; l_1, \ldots, l_s):\,& |k_i| > 0, \,0 \prec l_1 \prec \cdots \prec l_s,\\
    &\textstyle\sum_{i=1}^s k_i = \lambda, \,\mbox{and}\,\textstyle\sum_{i=1}^s |k_i|l_i = \nu\}.
  \end{split}
  \end{equation}
  Here, $l_i\prec l_j$ reflects a linear ordering on multiindices $l_i$ and $l_j$ (see~\cite{constantine1996multivariate} for details, unused here).

  Since $\Gamma$ is piecewise-$C^{m+2}$ it follows from
  \Cref{cor:curvilinear} that $\left|D^\alpha_{\widehat{x}}
  [T_\ell(\widehat{x})]_i\right| \le c_3 h^{|\alpha|}$ for $|\alpha| \le m+2$ and
  $1 \le i \le 3$. This component-wise bound, together with the
  property $\sum_{i=1}^s |k_i| l_i = \nu$ of elements in $p_s(\nu, \lambda)$ expressed in~\cref{eq:psdef} imply that $|[D^{l_j}_{\widehat
  y} T_\ell]^{k_j}| \le C h^{|\nu|}$, and substituting this bound
  into~\cref{eq:faa_di_bruno} and using~\cref{eq:interpolation_estimate} from
  \Cref{taylor_lagrange_interpolation_lemma}, which is valid since $y \in
  K_\ell \subset \Omega \setminus B_\delta(x)$, the result follows.$\qed$
\end{proof}

The following lemma, whose proof is given in Section~\ref{proof:ordquad_convergence_farfield_optimal},
establishes optimal convergence rates for sufficiently-powerful quadratures ($m$ sufficiently large). Here, optimality refers to matching the error introduced by
approximating $f$ by $f_n$ on $K$: $\left|\Zcal[f-f_n](x) - \ZcalB[f-f_n](x)\right| \lesssim h^{n+4-\sigma}$ per~\cref{Vk_analytical_regularization_error} (or~\cref{Xbreve_regularized}). It does so by developing bounds
for the  quadrature error over $\Omega\setminus\Ncal_h$ with a sequence of
shells (see~\cite[\S 6.1]{anderson2024fast} for a discussion of the necessity of ring/shell-like arguments to obtain
optimal rates).\enlargethispage*{5ex}

\begin{lemma}\label{lem:ordquad_convergence_farfield_optimal}
Take as given the assumptions and setting of \Cref{thm:tri_error_analysis}. In particular, $n\in\N_0$ is the interpolation degree and $m\in\N$, $m > n$, is the degree of exactness of the reference quadrature rule leading to the global composite quadrature rule $\Qcal^{h,m}_\Omega$ in~\eqref{eq:globnodesweights}-\eqref{eq:standard_QR}, and $\Tcal = (\Tcal_h)_{0<h\le h_0}$ is a family of meshes for which \Cref{cor:curvilinear} holds with $\theta \ge m + 1$. For all $x \in \Ecal_h$ and for sufficiently small $h>0$, it holds that
\begin{equation}\label{far_quadrature_optimal_estimate_V}
  \bigg|\, \int_{ \Omega \setminus \Ncal_h} \phi(x,y)\de y - \Qcal_{\Omega \setminus \Ncal_h}[\phi(x, \cdot)] \,\bigg| \le \Ccal \Lambda
        \begin{cases}
            h^{n+4-\sigma}, & m > n + 3 - \sigma \\[-0.5ex]
            h^{n+4-\sigma}\left|\log h\right|, &m = n + 3 - \sigma\\[-0.5ex]
            h^{m+1}, & m < n + 3 - \sigma,
        \end{cases}
\end{equation}
with positive constant $\Ccal = \Ccal(m,n, k, \kappa_\Tcal, q_\Tcal, c_{m+2}(\Tcal), f, \Omega)$ independent of $h$ and $\Lambda$.
\end{lemma}

\proofstep{Proof completion.}
For all $x \in  K\cap\Ecal_h$ it holds that
\begin{equation*}\begin{split}
  \lnrm \Zcal[f-f_n](x) - \ZcalB^{h,m}[f-f_n](x) \rnrm &\le C_{p,d}\lpar\lnrm \Zcal[f-f_n](x) - \ZcalB[f-f_n](x) \rnrm\\
      &+ \lnrm \ZcalB[f-f_n](x) - \ZcalB^{h,m}[f-f_n](x) \rnrm\rpar,
      \end{split}
\end{equation*}
with $f_n=f_n(\cdot,K)$, and the result follows by collecting~\cref{Vk_analytical_regularization_error} (or~\cref{Xbreve_regularized} in the case of $\Zcal=\Xcal$), \cref{near_quadrature_estimate}, and~\cref{far_quadrature_optimal_estimate_V} from \Cref{lem:ordquad_convergence_farfield_optimal}---the latter inequality applied to each pairwise combination of vector/tensor components of $f - f_n$ and of $\sfK$.\enlargethispage{1ex}
\end{proof}

\subsection{Proof of \Cref{taylor_lagrange_interpolation_lemma}}
\label{proof:taylor_lagrange_interpolation_lemma}

Our goal is to find suitable bounds for $D^{\alpha}_{y}\phi(x,y)$, which, applying the product rule, can be expressed as
  \begin{equation}\label{leibniz_rule_phi1}
    D^{\alpha}_{y}\phi(x,y) = \sum_{\beta: \beta \leq \alpha} \binom{\alpha}{\beta} \left(D^\beta_{y} \sfK(x,y)\right)\left(D^{\alpha-\beta} (f(y)-f_n(y;K))\right).
  \end{equation}

Our proof strategy to bound the other factors in~\eqref{leibniz_rule_phi1} is based on expressing
  \begin{equation}\label{eq:decomp_taylor_lagrange}
      f(y) - f_n(y; K) = R_{x, n}^{(1)}[f](y) + R_{x, n}^{(2)}[f](y),
  \end{equation}
where, denoting by $\mathscr{T}_{x, n}[f](y)$ the $n$th total degree Taylor
    polynomial of $f$ centered at $x$, we define the two remainder functions
  \begin{equation}\label{eq:decomp_taylor}
      R_{x, n}^{(1)}[f](y) := f(y) - \mathscr{T}_{x,n}[f](y),\qquad
      R_{x, n}^{(2)}[f](y) := \mathscr{T}_{x,n}[f](y) - f_n(y; K).
  \end{equation}
The term $R_{x, n}^{(1)}[f]$ can be simply expressed in terms of the Taylor remainder, while $R_{x, n}^{(2)}[f]$ will be rewritten using the `multipoint' Taylor formula for Lagrange polynomials~\cite{Ciarlet:72}.

As expected, for $\eta\in\N_0^3$, the Taylor remainder $R_{x, n}^{(1)}[f](y)$ satisfies
\begin{equation}\label{eq:taylor_error_bounds}
      \left|  D^{\eta}_{y} R^{(1)}_{x, n}[f](y)\right| \le
      \begin{cases}
          C_{R^{(1)}} |y - x|^{n + 1 - |\eta|},  & n + 1 > |\eta|, \\[-0.5ex]
          C_{R^{(1)}}, &n + 1 \le |\eta| \leq m + 1,
      \end{cases}
\end{equation}
where $C_{R^{(1)}} = C_{R^{(1)}}(\eta, n, f) > 0$. This estimate can be proven for $|\eta|<n+1$   by first leveraging the identity $ D^{\eta} R^{(1)}_{x, n}[f] =R^{(1)}_{x, n-|\eta|}[D^\eta f] $    which follows directly from the fact that $D^\eta \mathscr{T}_{x,n}[f]$ is the  $(n-|\eta|)$th-degree Taylor polynomial  of $D^\eta f$. From the integral form of the Taylor remainder $R^{(1)}_{x, n-|\eta|}[D^\eta f]$~\cite[Sec.1.1]{taylor2013partial}, we obtain
\[
  D^{\eta}_{y} R^{(1)}_{x, n}[f](y)
  = (n\shp 1\shm |\eta|)\sum_{\gamma:|\gamma| = n+1-|\eta|}  \frac{(y\shm x)^\gamma}{\gamma!}\int_0^1 (1\shm t)^n D^{\gamma+\eta} f(x \shp t(y\shm x))\de t.
\]
(For the sake of simplicity of presentation, we have assumed here that $x$ and $y$ can be joined by a straight line, an assumption that is fulfilled for convex domains $\Omega$. However, a simple generalization of the remainder formula \cite[Thm.A1]{driveranalysis}, which is valid for path-connected domains, can be used instead.)
Then, from the bounds
$$\left|\int_0^1 (1 - t)^n  D^{\gamma+\eta} f(x + t(y - x))\de t \right|\leq \frac{\|f\|_{C^{|\gamma|+|\eta|}(\overline\Omega)}}{n+1}$$
and $|(y-x)^{\gamma}|\leq \prod_{j=1}^3|y-x|^{\gamma_j} = |y-x|^{|\gamma|}$,
we readily arrive at
\begin{equation*}
      \left| D^{\eta}_{y} R^{(1)}_{x, n}[f](y)\right|\leq \bigg\{\|f\|_{C^{n+1}(\overline\Omega)}\left(\frac{n+1-|\eta|}{n+1}\right)
          \sum_{\gamma:|\gamma| = n + 1-|\eta|} \frac{1}{\gamma!} \bigg\} |y-x|^{n+1-|\eta|},
\end{equation*}
which gives an expression for the constant $C_{R^{(1)}}$ in the case $|\eta|\leq n+1$. For $n+1<|\eta|\leq m+1$, on the other hand, the bound follows directly from the fact that $D^\eta \mathscr{T}_{x,n}[f]=0$, which allows us to select $C_{R^{(1)}}=\|f\|_{C^{|\eta|}(\overline\Omega)}$.

Remembering that $\alpha\in\N_0^3$ is a multi-index satisfying $|\alpha| \leq m + 1$ and utilizing the bounds~\eqref{eq:taylor_error_bounds} with $\eta=\beta$ in conjunction with bounds on admissible kernels given in \Cref{def:kernel_bounds} with $\eta = \alpha - \beta$, we have
\begin{align}\label{eq:prod_Taylor}
\MoveEqLeft[3]
        \left| D^{\alpha}_{y} \left(\sfK(x,y)R_{x, n}^{(1)}\vphantom{D^{\alpha}_{y} (\sfK(x,y)R_{x, n}^{(1)}}[f](y) \right)\right|
      = \bigg| \sum_{\beta: \beta \le \alpha} \binom{\alpha}{\beta} \left(D_{y}^{\alpha-\beta} \sfK(x,y)\right) \left(D^{\beta}_{y} R_{x, n}^{(1)}[f](y)\right)\bigg| \\[-1ex]
      &\le \sum_{\beta: |\beta| < n+1} \binom{\alpha}{\beta} C_{\sfK,|\alpha|,\operatorname{diam} \Omega} \,|x - y|^{-\sigma - |\alpha-\beta|} C_{R^{(1)}} |x - y|^{n+1-|\beta|} \notag \\[-1ex] & \qquad
        \phantom{\le} +\sum_{\beta: \beta \le \alpha,\; n+1\le|\beta|} \binom{\alpha}{\beta} C_{\sfK,|\alpha|,\operatorname{diam} \Omega} |x - y|^{-\sigma -|\alpha-\beta|} C_{R^{(1)}}. \notag
\end{align}
To bound the terms in~\eqref{eq:prod_Taylor}, we note that since $|x-y|\geq\delta$, we have $|x - y|^p \leq  \delta^p$ for any $p < 0$; conversely, if $p > 0$ then $|x - y|^p \leq C_{\Omega, p}$. Then, since $\beta\leq\alpha$, $|\alpha-\beta|=|\alpha|-|\beta|$,
\begin{equation*}
\begin{split}
  \sum_{\beta: |\beta| < n+1} \binom{\alpha}{\beta} C_\sfK &|x - y|^{-\sigma -|\alpha-\beta|} C_{R^{(1)}} |x - y|^{n+1-|\beta|}\lesssim |x-y|^{-\sigma + n + 1 -|\alpha|}\\[-1ex]
  &\le \begin{cases} C_{|\alpha|, \sfK, R^{(1)}} \delta^{-\sigma + n + 1- |\alpha|}, &-\sigma + n + 1 -|\alpha| < 0\\
                     C_{\Omega, |\alpha|, \sfK, R^{(1)}}, &-\sigma + n + 1 -|\alpha| \ge 0
  \end{cases}.
  \end{split}
\end{equation*}
For the second right-hand-side term in~\cref{eq:prod_Taylor}, if $|x - y| \geq \delta > 1$, then the sum is bounded by a constant dependent only on $n$, $|\alpha|$, and the diameter of $\Omega$, while if $0 < \delta < 1$, since $|\beta| \geq n + 1$, we have
\begin{equation*}
  \sum_{\beta: \beta \le \alpha,\; n+1\le|\beta|} \binom{\alpha}{\beta} C_{\sfK,|\alpha|,\operatorname{diam} \Omega} |x - y|^{-\sigma -|\alpha-\beta|} C_{R^{(1)}}
  \lesssim \delta^{-\sigma - |\alpha| + n + 1},
\end{equation*}
where the implied constant in the last bound depends on $|\alpha|,n,\OO$.

Therefore, using the above relations we conclude that~\eqref{eq:prod_Taylor} simplifies to
    \begin{equation}\label{eq:taylor_error_final}
      \left| D^{\alpha}_{y} \left(\sfK(x,y)R_{x, n}^{(1)}[f](y) \right)\right|  \le
      C\subT\delta^{-\sigma + n + 1 - |\alpha|} ,\quad y\in \Omega\setminus B_\delta(x), \ \delta>0.
\end{equation}
where $ C\subT = C\subT(|\alpha|,n,f,\Omega, C_{K,|\alpha|})$. This completes the estimate concerning $R^{(1)}_{x,n}[f]$.

Turning to $R^{(2)}_{x,n}[f]$, we exploit the polynomial interpolation estimate~\cite[Eq.\ (6.43)]{anderson2024fast}, developed in Lemma 6.8 of~\cite{anderson2024fast} and  based on the fundamental work of~\cite{Ciarlet:72}, which, relying on the mesh being shape-regular and quasi-uniform (Def.~\ref{def:regularuniform}), bounds differences between a Lagrange interpolating polynomial on $K$ and a Taylor polynomial developed at $x$.  For a multi-index $\eta\in\N_0^3$ and for all $y \ne x$, this estimate provides that, $c_\gamma$ being~\cite[eq.~(6.38)]{anderson2024fast} the coefficients in $R^{(2)}_{x,n}[f](y) = \sum_{\gamma:|\gamma|\le n} c_\gamma(y - x)^\gamma$,
\begin{align}\label{eq:lagrange_taylor_error}
  \left| D^{\eta}_{y} R^{(2)}_{x,n}[f](y)\right|
  &= \left|D^{\eta}_{y} \left(\mathscr{T}[f]_{x,n}(y) - f_n(y; K)\right)\right|\\
        &\le \sum_{\gamma:|\gamma| \le n,\; \eta\leq\gamma}  \frac{\gamma!}{(\gamma - \eta)!} \left|c_\gamma(y - x)^{\gamma-\eta}\right| \notag \\
        &\le C_{R^{(2)}} \Lambda \sum_{\gamma:|\gamma| \le n,\; \eta\leq\gamma} h^{n + 1 - |\gamma|} \left|y - x\right|^{|\gamma| - |\eta|}, \notag
\end{align}
for $|\eta|\leq n$, where $C_{R^{(2)}} = C_{R^{(2)}}(n, \eta, \kappa_\Tcal, q_{\Tcal}, f, \Omega) > 0$, and $\labs D^{\eta}_{y} R^{(2)}_{x, n}[f](y) \rabs=0$ otherwise. This bound makes use of the assumption $f \in C^{n+3}(\overline\Omega)$ together with a Sobolev extension theorem~\cite[\S 6, Thm.\ 5]{stein1970singular} to ensure that a suitable extension, also denoted by $f$, satisfies $f \in C^{n+1}(H(K))$; note that the bound is proven for curved elements yet is quantitatively dependent only on $\Lambda$, $\kappa_\Tcal$, and $q_\Tcal$. Again utilizing the bounds~\eqref{eq:greens_admissible} together with~\eqref{eq:lagrange_taylor_error} we have, for $y \ne x$,
\begin{subequations}
\begin{align}\label{eq:lagrange_multipoint_error_final}
\MoveEqLeft[1]
        \left| D^{\alpha}_{y} \left(\sfK(x,y)R_{x, n}^{(2)}[f](y) \right)\right|
        = \bigg| \sum_{\beta: \beta \le \alpha} \binom{\alpha}{\beta} \left(D_{y}^{\alpha-\beta} \sfK(x,y)\right) \left(D^{\beta}_{y} R_{x, n}^{(2)}[f](y)\right)\bigg| \\
          &\le \sum_{\substack{\beta: \beta \le \alpha,\\ |\beta|\leq n}} \binom{\alpha}{\beta} C_{\sfK,|\alpha|,\operatorname{diam} \Omega} |y\shm x|^{-\sigma - |\alpha|+|\beta|} C_{R^{(2)}} \Lambda\sum_{\substack{\gamma:|\gamma| \le n\\\beta\leq\gamma}} h^{n + 1 - |\gamma|} |y\shm x|^{|\gamma| -  |\beta|} \notag \\
          &\le  C\subL \Lambda \sum_{\gamma:|\gamma| \le n} h^{n + 1 - |\gamma|} |y\shm x|^{-\sigma + |\gamma| - |\alpha|}, \notag
\end{align}
where $C\subL = C\subL(|\alpha|,n, \kappa_\Tcal, q_\Tcal, f, \sfK, \Omega) > 0$.

For $|x-y|\geq\delta$, the above estimate, which holds for all $x,y\in\Omega$, $x\neq y$, simplifies for $-\sigma + |\gamma| - |\alpha| <0$ to
\begin{equation}\label{eq:lagrange_multipoint_error_final_2}
  \left| D^{\alpha}_{y} \left(\sfK(x,y)R_{x, n}^{(2)}[f](y) \right)\right|
  \le  C\subL \Lambda \sum_{\gamma:|\gamma| \le n} h^{n + 1 - |\gamma|} \delta^{-\sigma + |\gamma| - |\alpha|},
\end{equation}
while for $-\sigma + |\gamma| - |\alpha| \geq 0$ the same estimate, up to a suitable redefinition of $C\subL$, holds since $|x - y| \leq \operatorname{diam}(\Omega)$.

\label{eq:both_esti}
\end{subequations}
Estimates~\eqref{eq:taylor_error_final} and~\eqref{eq:both_esti} imply~\eqref{eq:interpolation_estimate}.\enlargethispage*{1ex}

\subsection{Proof of \Cref{lem:ordquad_convergence_farfield_optimal}}
\label{proof:ordquad_convergence_farfield_optimal}

Let $x \in \Ecal_h\subset\overline{\Omega}$ be given such that $K\ni x$. Let $\delta_0\leq \operatorname{diam}\Omega$ be a fixed length and assume without loss of generality that $h_0 < \delta_0$ so that $h\le h_0<\delta_0$.     Consider a sequence of radii $r_j = jh$, ($j=0,\ldots,J+1$), with the integer $J$ such that $Jh<\delta_0\leq (J+1)h$), and let
$$
    \widetilde{A}_{j}=\widetilde{A}_j(x) := \{ y\in\Omega: jh \le |x -y| \le (j+1)h \}, \qquad j=0,\ldots,J,
$$
in order to define the `meshed' shells $A_j=A_j(x)$ (reference~\cite[Fig.\ 3]{anderson2024fast} displays the setup in the analogous 2D case):
\begin{equation}
\left\{\begin{aligned}
 A_0 &:= \cup\,\big\{K_\ell\in\Tcal_h:K_\ell\cap \widetilde{A}_0\neq\emptyset\big\},\\
 A_j &:= \cup\,\big\{K_\ell\in\Tcal_h:K_\ell \cap \widetilde{A}_j\neq \emptyset, K_\ell \not\subset A_{j-1}\big\}  &\quad&(1\leq j\leq J),\\
 A_{J+1} &:= \overline{\Omega}\setminus\cup_{j=0}^J\, A_j. 
\end{aligned}\right.
\label{eq:meshedring_def}
\end{equation}
Note that with the definitions above we have $A_0=\Ncal_h$ and $\bigcup_{j=1}^{J+1}A_j = \overline\Omega\setminus\Ncal_h$  (recall definition~\eqref{Nh:def} of $\Ncal_h$). By the triangle inequality it then follows that
\begin{multline}\label{eq:inte_split}
  \bigg|\int_{\Omega\setminus\Ncal_h} \phi(x,y)\de y - \Qcal_{\Omega \setminus \Ncal_h}[\phi(x, \cdot)] \bigg|
  \leq  \sum_{j=1}^{J} \Rcal(j) + \Rcal(J\shp 1) \\
  \text{with} \quad \Rcal(j) := \bigg|\int_{A_j} \phi(x,y)\de y - \Qcal_{A_j}[\phi(x, \cdot)]\bigg|.
\end{multline}
In what follows we derive error estimates for each summand of~\eqref{eq:inte_split} which, for convenience, we write as a collection of reference space integrals:
\begin{equation}\label{eq:Aj_reference}
  \int_{A_j} \phi(x, y)\de y = \sum_{K_\ell \in A_j} \int_{\KHat} \phi(x, T_\ell(\widehat{y})) |\operatorname{det} J_\ell(\widehat{y})|\de \widehat{y}, \quad j = 1, \ldots, J+1.
\end{equation}
Since by quasi-uniformity
and shape-regularity and the fact that, for $1 \le j \le J$, $jh \le |x - y| < (j + 2)h$ for every $y \in A_j$, we have $\#\{K_\ell: K_\ell \subset A_j\} \le C_{q_\Tcal,
\kappa_\Tcal} j^2$ so that, calling $\psi(x, \widehat{y}) \coloneqq \phi(x, T_\ell(\widehat{y})) |\operatorname{det} J_\ell(\widehat{y})|$ and applying~\cite[Eq.\ (7.21)]{IsaacsonKeller} to $\widehat{\mathcal{Q}}^m[\psi]$, we have
\begin{multline}\label{eq:single_ring_bound}
   \hspace*{-0.5em}\Rcal(j) =
    \Bigg|\sum_{K_\ell \in A_j}\!\! \bigg\{\! \int_{\KHat}  \phi(x, T_\ell(\widehat{y})) |\operatorname{det} J_\ell(\widehat{y})|\de \widehat{y} - \sum_{i=1}^{q_n} \widehat{\omega}_i \phi(x, T_\ell(\widehat{y}_i))\left|\det J_\ell(\widehat{y}_i)\right|\,\bigg\}\Bigg| \hspace*{-2em} \\
    \le C_1 j^2 M_{m+1}(j)
\end{multline}
where $C_1 = C_1(m,q_\Tcal,
    \kappa_\Tcal)$ and $M_{m+1}(j)$ is a constant such that
\begin{equation}\label{eq:Mbound_def}
  \sup_{\ell: K_\ell \subset A_j} \sup_{\widehat{y} \in \KHat} \sup_{|\alpha| = m + 1} \left|D^\alpha_{\widehat{y}} \psi(x, \widehat{y}) \right| \le M_{m+1}(j).
\end{equation}
Such bounds together with~\cref{eq:single_ring_bound} will serve as a composite quadrature error estimate for each of the integrals over $A_1, \ldots, A_J$ in~\cref{eq:inte_split}; we will likewise develop, later, quadrature error bounds for the integral over $A_{J+1}$.

To find such constants $M_{m+1}(j)$, we begin by observing that, in view of the strict sign-definiteness of $\det J_\ell$, we may express
\begin{equation}\label{eq:mapped_regularized_prod_rule}
  \left|D^\alpha_{\widehat{y}} \psi(x, \widehat{y})\right| = \bigg|\sum_{\beta: \beta \le \alpha}\binom{\alpha}{\beta} \left((D^\alpha_{\widehat{y}} \phi(x, T_\ell(\widehat{y}))\right)(D^{\alpha - \beta}_{\widehat{y}} \operatorname{det} J_\ell(\widehat{y}))\bigg|.
\end{equation}
It follows from \Cref{cor:curvilinear} that
\begin{equation}\label{eq:deriv_jacobian}
  |D^{\alpha - \beta}_{\widehat{y}} \det J_\ell(\widehat{y})| \le C_2 h^3\quad\mbox{for}\quad |\alpha - \beta| \le m+1; \quad C_2 = C_2(\Gamma, c_{m+2}(\Tcal), m),
\end{equation}
since, calling the components of $T_\ell = [T_{\ell, 1}, T_{\ell, 2}, T_{\ell, 3}]$, and since $\det J_\ell$ is a sum of 6 terms of the form
\[
  \left(\partial_{\widehat{x}_q}T_{\ell, i}\right) \left(\partial_{\widehat{x}_r} T_{\ell, j}\right) \left(\partial_{\widehat{x}_s} T_{\ell, k}\right), \quad\mbox{with}\quad i, j, k, q, r, s \in \{1, 2, 3\}.
\]
Using product rule and~\cref{eq:Dalpha_TK} the bound follows---the maximum total derivative order of components $[T_\ell]_i$ being $m+2$, all of which exist and are bounded per~\cref{eq:Dalpha_TK} since by assumption $\theta \ge m + 1$.\enlargethispage*{1ex}

But, letting $\ell$ be such that $A_\ell \in A_j$, then using the fact that for all $\widehat{y} \in \KHat$ and setting $y = T_\ell(\widehat{y})$, we have $jh \le |x - y|$, then choosing $\delta = jh$ in \Cref{taylor_lagrange_interpolation_lemma_reference},
\begin{multline}
    \sup_{\ell: K_\ell \subset A_j} \sup_{\widehat{y} \in \KHat} \sup_{|\beta| = m + 1} \left|D^\beta_{\widehat{y}} \phi(x, T_\ell(\widehat{y}))\right| \\
  \le h^{m+1} \sum_{\lambda: 1 \le |\lambda| \le m+1} (jh)^{-\sigma + n + 1 - |\lambda|} \Big( \widehat{C}\subT
  + \widehat{C}\subL \Lambda S_n(1/j) \Big).
  \end{multline}
Simplifying and using~\cref{eq:deriv_jacobian}, we find a bound for~\cref{eq:Mbound_def} of
\[
    M_{m+1}(j) = \sum_{\lambda: 1 \le |\lambda| \le m+1} h^{-\sigma + n + 4 - |\lambda| + (m+1)}
    j^{-\sigma + n + 1 - |\lambda|} \Big( \widehat{C}\subT + \widehat{C}\subL \Lambda S_n(1/j) \Big).
\]
Using this derivative bound, we sum the inequality~\cref{eq:single_ring_bound} over the $J$ shell domains and find
\begin{multline}\label{eq:composite_estimate_1}
    \sum_{j=1}^J \Rcal(j) \le \sum_{j=1}^J C_1 j^2 M_{m+1}(j) \\
    \le C_1 \sum_{\lambda: 1\le|\lambda| \le m+1} h^{-\sigma + n + 4 - |\lambda| + (m+1)} j^{-\sigma + n + 3 - |\lambda|} \Big( \widehat{C}\subT + \widehat{C}\subL \Lambda S_n(1/j) \Big).
\end{multline}
Letting $C_3 = C_3(\sigma, n, m, \delta_0)$ denote a constant and using $J < \delta_0 h^{-1}$, we will find useful the bound
\[
  \sum_{j=1}^J j^{-\sigma + n + 3 - |\lambda|}
  \le C_3
  \begin{cases}
      1, &\quad -\sigma + n + 3 - |\lambda| < -1,\\
      |\log h|, &\quad -\sigma + n + 3 - |\lambda| = -1,\\
      h^{-1 - (-\sigma + n + 3 - |\lambda|)}, &\quad -\sigma + n + 3 - |\lambda| \ge 0,
  \end{cases}
\]
the sum being bounded above, in the first case for $-\sigma + n + 3 - |\lambda| < -1$ by
a constant (as $\sum_{j=1}^\infty j^{-\zeta}$ converges for $\zeta > 1$), in the second case for
$-\sigma + n + 3 - |\lambda| = -1$ by the harmonic numbers $H_J$ that grow as
$\log J\lesssim |\log h|$, and lastly for $-\sigma + n + 3 - |\lambda| \ge 0$  using
$J<\delta_0 h^{-1}$ and Faulhaber's formula.
Similarly, we find
\[
  \sum_{j=1}^J  j^{-\sigma + 3 + |\gamma| - |\lambda|} S_n(1/j)
  \le C_3
  \begin{cases}
      1, &\quad -\sigma + n + 2 - |\lambda| < -1,\\
      |\log h|, &\quad -\sigma + n + 2 - |\lambda| = -1,\\
      h^{-1 - (-\sigma + n + 2 - |\lambda|)}, &\quad -\sigma + n + 2 - |\lambda| \ge 0.
  \end{cases}
\]
Substituting into~\cref{eq:composite_estimate_1}, one finds
\begin{align}\label{eq:composite_estimate_2}
    \sum_{j=1}^J \Rcal(j)
    &\le C_4 \Lambda \kern-1em \sum_{\lambda: 1\le|\lambda| \le m+1}
  \begin{cases}
      h^{-\sigma + n + 4 - |\lambda| + (m+1)}, &\quad -\sigma + n + 3 - |\lambda| < -1,\\[-0.5ex]
      h^{-\sigma + n + 4 - |\lambda| + (m+1)}|\log h|, &\quad -\sigma + n + 3 - |\lambda| = -1,\\[-0.5ex]
      h^{m + 1}, &\quad -\sigma + n + 3 - |\lambda| \ge 0,
  \end{cases} \notag\\
  &\quad+\sum_{\lambda: 1\le|\lambda| \le m+1}
  \begin{cases}
      h^{-\sigma + n + 4 - |\lambda| + (m+1)}, &\quad -\sigma + n + 2 - |\lambda| < -1,\\[-0.5ex]
      h^{-\sigma + n + 4 - |\lambda| + (m+1)}|\log h|, &\quad -\sigma + n + 2 - |\lambda| = -1,\\[-0.5ex]
      h^{m + 2}, &\quad -\sigma + n + 2 - |\lambda| \ge 0,
  \end{cases}\notag
\end{align}
with $C_4 = C_4(n, m, q_\Tcal,\kappa_\Tcal,\sigma,\delta_0, \widehat{C}\subT, \widehat{C}\subL, \Gamma)$, and finally,
\begin{equation}\label{eq:composite_estimate_final}
    \sum_{j=1}^J \Rcal(j) \le C_5 \Lambda
  \begin{cases}
    h^{-\sigma + n + 4}, &\quad -\sigma + n + 3 - (m+1) < -1,\\[-0.5ex]
      h^{-\sigma + n + 4}|\log h|, &\quad -\sigma + n + 3 - (m+1) = -1,\\[-0.5ex]
      h^{m + 1}, &\quad -\sigma + n + 3 - (m+1) \ge 0,
  \end{cases}
\end{equation}
with $C_5 = C_5(q_\Tcal,\kappa_\Tcal,\sigma,\delta_0, \widehat{C}_{\mathscr{T}}, \widehat{C}_{\mathscr{L}}, \Gamma)$.\enlargethispage*{1ex}

To estimate the quadrature error associated with the integral over
$A_{J+1}$, i.e.~\cref{eq:single_ring_bound} with $j = J+1$, we observe that
$\delta_0 < (J+1)h \le |y - x|$ for $y \in K_\ell \subset A_{J+1}$.  Therefore,
using $\delta = \delta_0$ in~\cref{eq:interpolation_estimate_reference}
together with~\cref{eq:deriv_jacobian} and substituting these into~\cref{eq:mapped_regularized_prod_rule}, we have
\begin{equation}\label{eq:deriv_bound_Jp1}
  M_{m+1}(J+1) = h^{m+4} \sum_{\lambda: 1 \le |\lambda| \le m+1} \delta_0^{-\sigma + n + 1 - |\lambda|} \Big( 
  \widehat{C}\subT
  + \widehat{C}\subL \Lambda S_n(h/\delta_0) \Big).
\end{equation}
Quasi-uniformity and shape-regularity of $\mathcal{T}$ ensure that $\#\{K_\ell: K_\ell \subset A_{J+1}\} h^3 \le C_{q_\Tcal, \kappa_\Tcal} \int_\Omega \de y$,
from which together with~\cref{eq:deriv_bound_Jp1} we have
\begin{equation}\label{eq:composite_estimate_Jp1}
  \Rcal(J\shp 1) \le C_6 \Lambda h^{m+1}.
\end{equation}
with $C_6 = C_6(q_\Tcal,\kappa_\Tcal,\sigma,\delta_0, \widehat{C}_{\mathscr{T}}, \widehat{C}_{\mathscr{L}}, \Gamma )$.
In view of~\cref{eq:inte_split}, from~\cref{eq:composite_estimate_final} and~\cref{eq:composite_estimate_Jp1} we have that the desired bound in~\cref{far_quadrature_optimal_estimate_V} is established with $\mathcal{C} =  C_5 + C_6$.\enlargethispage{3ex}

\section{Numerical demonstration}\label{sec:numer}

In this section we test the regularization for all four integral operators, $\Vcal$, $\nabla \Vcal$, $\Wcal$, and $\Xcal$, and we additionally demonstrate the success of the isogeometric mappings in two ways.

We begin by confirming the validity of the curved tetrahedron implementation without the use of volume integral operators. To do this we first compute the volume of the torus as the smoothness order $\theta$ and quadrature order $m$ vary. Because the map is exact, the only expected error for this test is that of quadrature with a smoothness-dependent convergence rate predicted by~\cref{cor:curvilinear}. \Cref{tab:curve_quad} shows the result of experiments of varying $\theta$, with $m = \theta$ kept fixed. We observe the expected order of convergence consistent with \Cref{cor:curvilinear}, except that superconvergence of exactly one order is observed for the case $m = \theta = 1$.
\begin{table}
    \centering
    \begin{tabular}{|c|c|c|c|c|c|c|c|c|}
    \hline
             & \multicolumn{2}{|c|}{$m = \theta = 1$} & \multicolumn{2}{|c|}{$m = \theta = 2$} & \multicolumn{2}{|c|}{$m = \theta = 3$} & \multicolumn{2}{|c|}{$m = \theta = 4$}\\
             \hline
         $h$ & er & eoc & er & eoc & er & eoc & er & eoc\\
         \hline
         0.24 & 7.58e-3 & -- & 3.23e-4 & --& 1.45e-5 & --& 2.23e-7 & --\\
         0.20 & 4.82e-3 & 2.48 & 2.15e-4 & 2.39& 6.99e-6 & 4.01& 1.05e-7 & 4.14\\
         0.16 & 2.43e-3 & 3.07&  1.09e-4 & 3.05& 2.79e-6 & 4.12& 3.45e-8 & 4.99\\
         0.12 & 1.01e-3 & 3.05&  4.59e-5 & 3.00& 7.85e-7 & 4.41& 8.26e-9 & 4.97\\
         0.08 & 3.02e-4 & 2.98&  1.38e-5 & 2.96& 1.43e-7 & 4.20& 1.10e-9 & 4.98\\
         0.04 & 3.86e-5 & 2.96&  1.79e-6 & 2.95& 8.40e-9 & 4.09& 1.99e-11 & 5.78\\
         0.02 & 4.82e-6 & 3.00&  2.25e-7 & 2.99& 4.95e-10 & 4.09& 1.56e-11 & 0.34\\
         \hline
    \end{tabular}
    \caption{Demonstration of `exact' curved tetrahedral element quadratures. The test computes the volume of the torus of major radius $r_1 = 1.0$ and minor radius $r_2 = 0.5$ using the curved tetrahedral elements implemented in \texttt{Inti.jl} and compares to the true value of $2\pi^2 r_2^2 r_1$; the table reports the error (er) and estimated order of convergence (eoc). The degree of smoothness $\theta$ is kept fixed to the quadrature order $m$ to saturate the quadrature order power.}
    \label{tab:curve_quad}
\end{table}

The methodology is tested by utilizing manufactured solutions. We begin with the operator $\Vcal$. For scalar problems we select a function $u: \Omega \to \mathbb{C}^p$ and set $f: \Omega \to \mathbb{C}^p$ by $f \coloneqq \mathcal{L} u$; specifically, $f(x) = (k_0^2 - k^2) \cos(k_0 (x \cdot p)), |p| = 1$. For Laplace ($k = 0$), $k_0 = \pi$ for $n = 1, 2$ and $k_0 = 5\pi$ for $n = 3, 4$, while for Helmholtz with $n = 1, 2, 3$ we selected $k = \pi$ and $k_0 = 2\pi$, while for $n = 4$, we selected $k = 2\pi$ and $k_0 = 3\pi$; in each case the parameters were chosen to be sufficiently challenging to demonstrate the convergence behavior within available resources. By Green's third identity we have that $\mathcal{V}[f] = v_{\mathrm{ref}}$ with the reference potential given by the formula
\[
v_{\mathrm{ref}}(x) \coloneqq -u(x) + \Scal[\Bcal_\nu u](x) - \Dcal[u](x).
\]
Here, $\Scal$ and $\Bcal$ are defined in~\cref{eq:single_double_layers}. Errors are reported as the maximum error from $v_{\mathrm{ref}}$ measured over the volume integration nodes and are displayed in Figures~\ref{fig:laplacehelmholtz}.   The convergence rates predicted from \Cref{thm:tri_error_analysis}, in this case for a kernel $\sfK$ admissible of class $\sigma = 1$, are displayed in the plots, and the numerical experiments are consistent with the theoretical results; it may be useful to refer to \Cref{3D-VR:m} for the tabulated relationship between quadrature and interpolation orders for the quadratures and interpolation node sets used in these experiments, which are always those of the Vioreanu-Rokhlin rules.\enlargethispage*{5ex}

\begin{figure}[t]
    \centering
    \begin{subfigure}[b]{0.45\textwidth}
        \includegraphics[width=\textwidth]{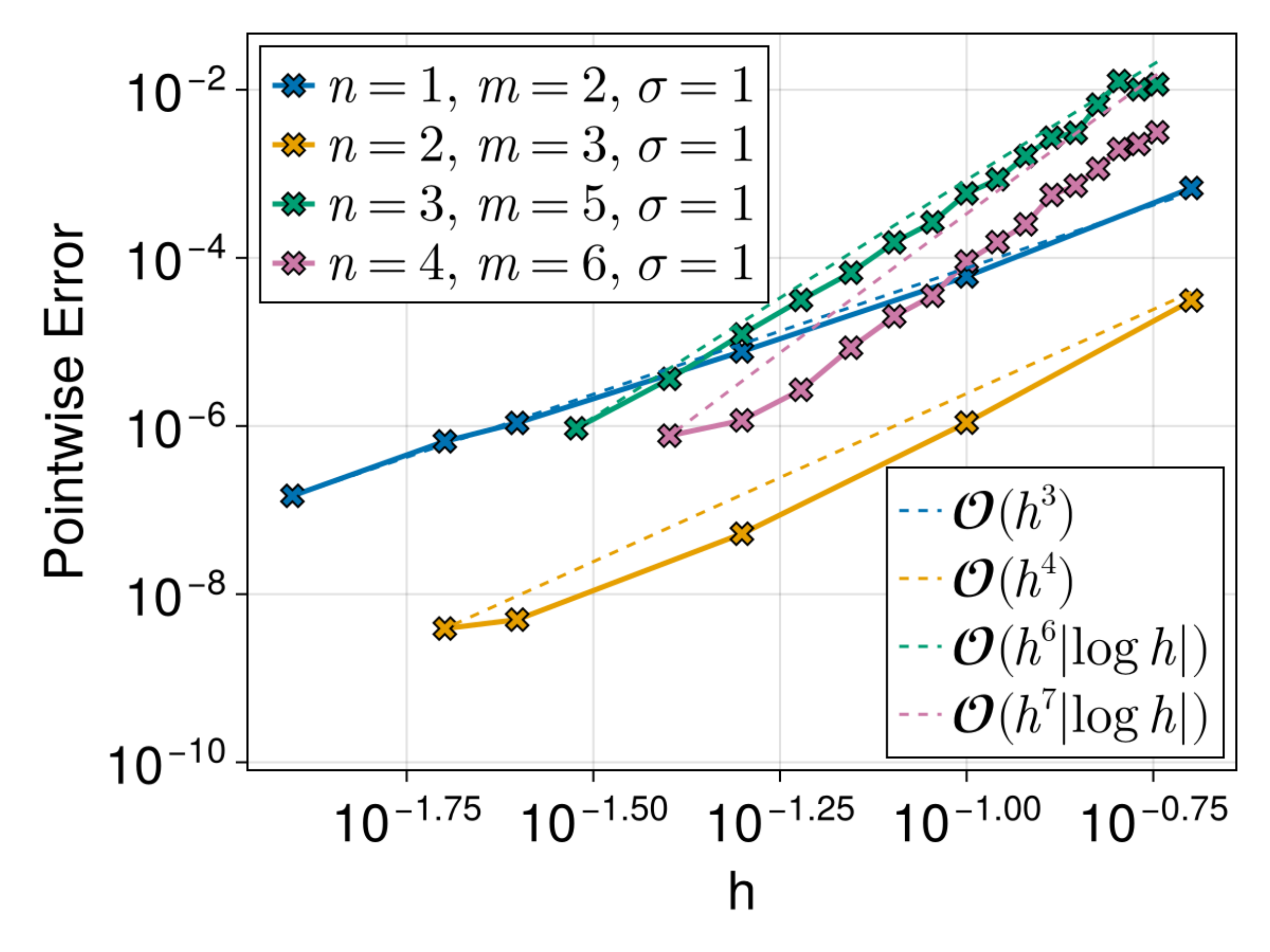}
    \end{subfigure}
    \begin{subfigure}[b]{0.45\textwidth}
        \includegraphics[width=\textwidth]{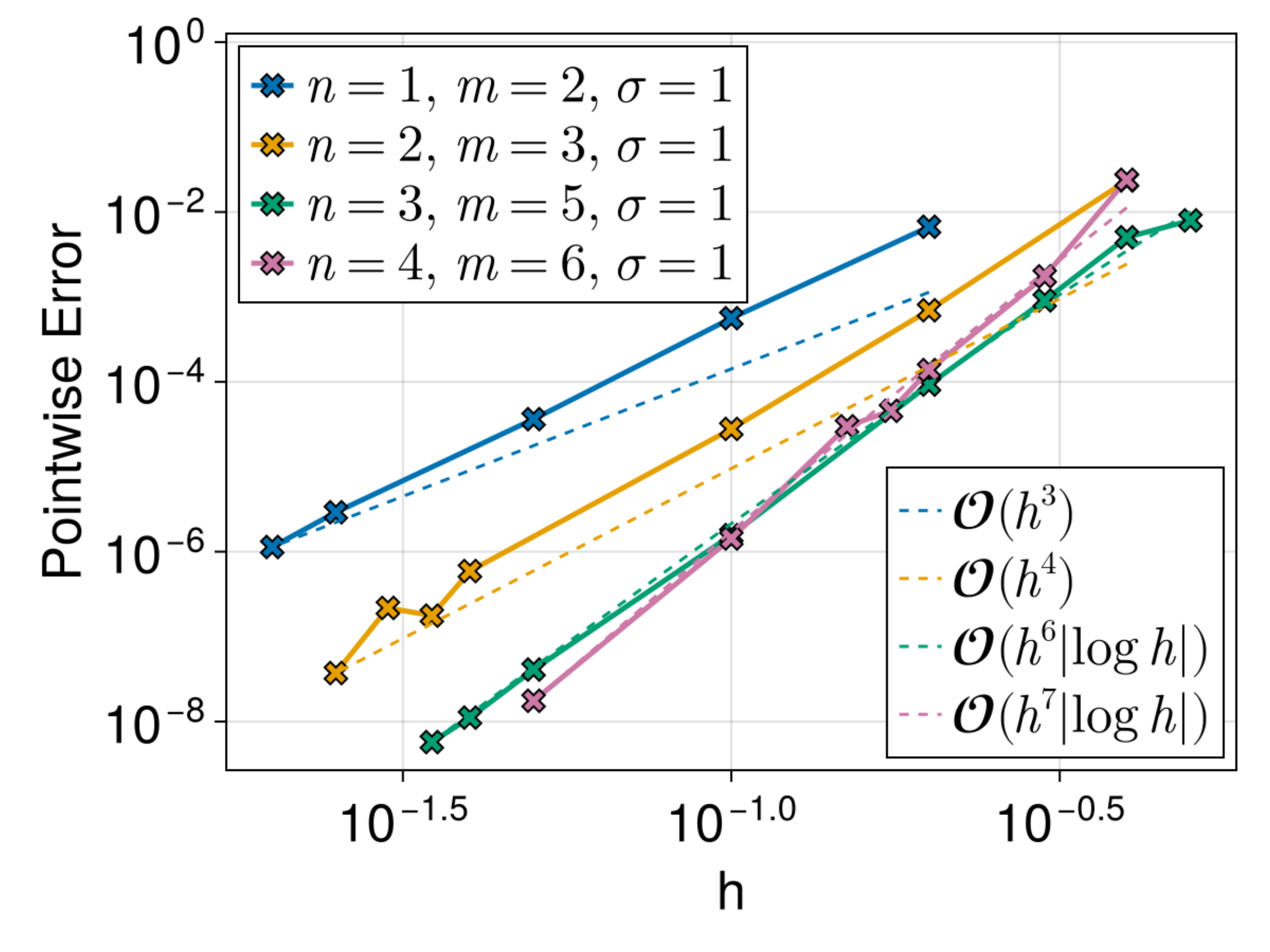}
    \end{subfigure}
     \caption{Green's identity convergence for Laplace (left) and Helmholtz (right) VIO $\Vcal$ on a sphere of unit radius.}
     \label{fig:laplacehelmholtz}

    \centering
    \begin{subfigure}[b]{0.45\textwidth}
        \includegraphics[width=\textwidth]{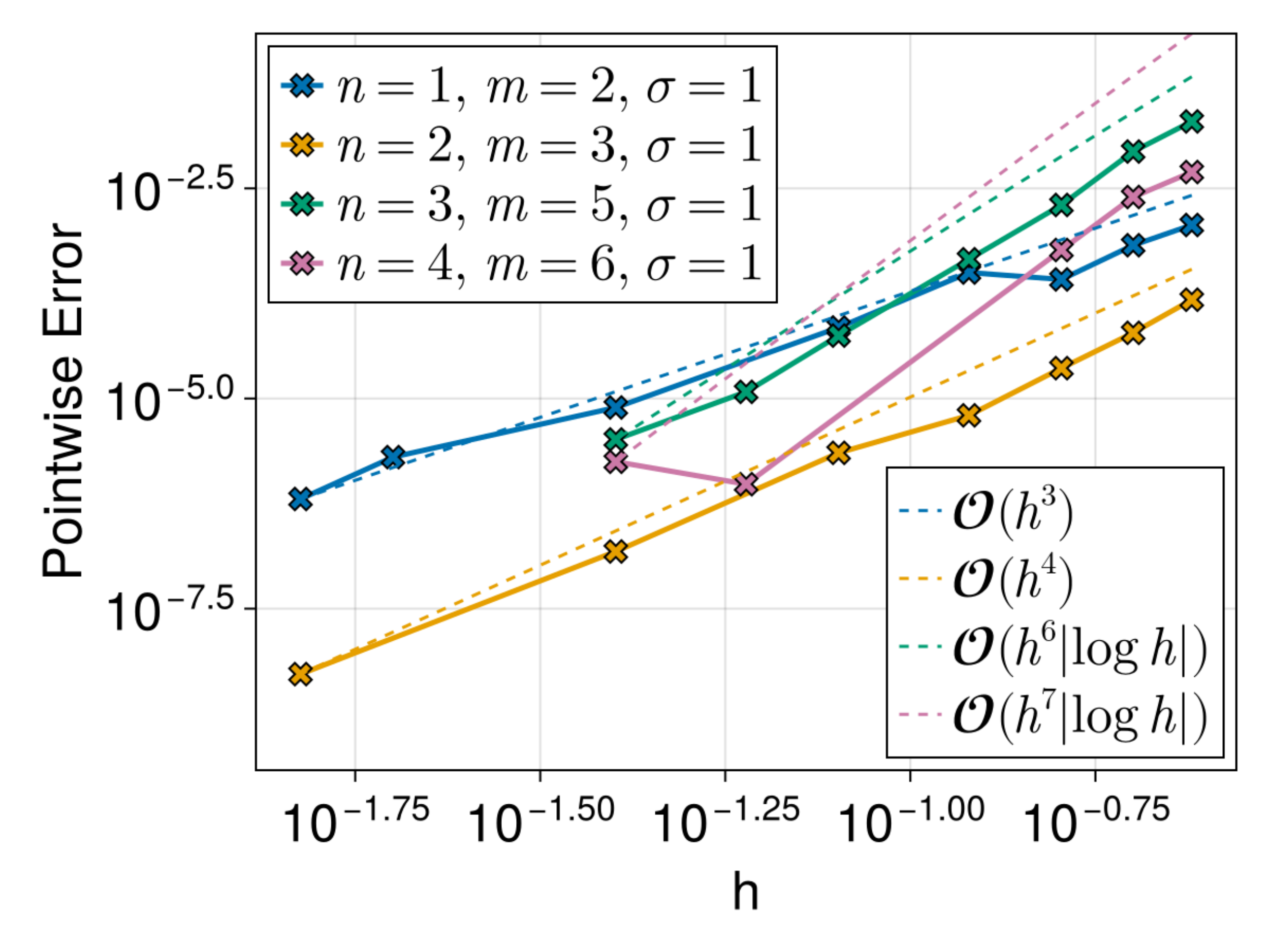}
    \end{subfigure}
    \begin{subfigure}[b]{0.45\textwidth}
        \includegraphics[width=\textwidth]{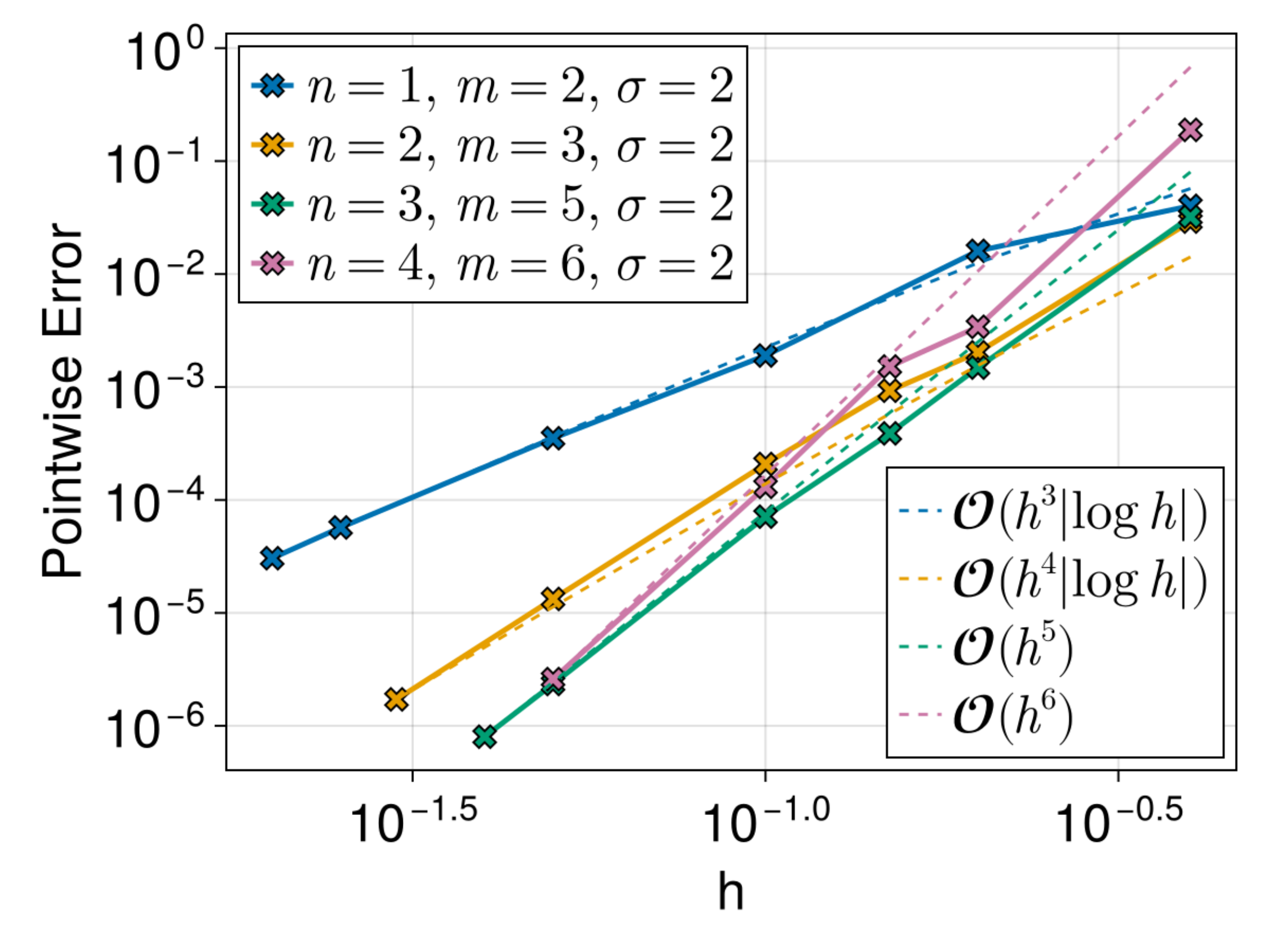}
    \end{subfigure}
    \caption{Poisson solution convergence on a torus (left) and Green's identity convergence for a Laplace VIO $\nabla \Vcal$ on a unit sphere (right).}
     \label{fig:poissongradlaplace}
\end{figure}

We test the operator $\nabla\Vcal$ in a similar way, comparing to the reference solution
\[
v_{\mathrm{ref}}(x) \coloneqq -\nabla u(x) + \nabla \Scal[\Bcal_\nu u](x) - \nabla \Dcal[u](x)
\]
for the Laplace PDE. For this problem the kernel $\sfK$ is admissible of class $\sigma = 2$. The experiments make use of sufficiently high-order methods~\cite{faria2021general} for evaluation of the surface potentials $\nabla \Dcal$ and $\nabla \Scal$ inherent in evaluating $\nabla \Vcal$ so as to match the order expected from the theory. The same manufactured solutions as before are used, making the selection for $n = 1, 2$ of $k = \pi$ while for $n = 3$ that of $k = 3\pi/2$ and for $n = 4$,  $k = 5\pi/2$; once again, the parameters were chosen to be sufficiently challenging to demonstrate the convergence behavior within available resources. Results are shown in Figure~\ref{fig:poissongradlaplace} and closely match the theoretical predictions.

Tests based on the Green's identities such as those presented so far can show high accuracy even in the presence of significant geometrical error, so we use a test of solving a Poisson boundary value problem as a more complete test of the isogeometric curved elements. The curved mapping smoothness parameter $\theta = m+1$, for each choice of quadrature order $m$ (see \Cref{3D-VR:m}), is chosen for purposes of satisfying the assumptions of \Cref{thm:tri_error_analysis}. We use a manufactured solution of $u_e(x) = \cos(k x_1) \sin(k x_2)$ on a torus $\Omega$ of minor radius $1/2$ and major radius of $1$; the boundary data used in solving the boundary value problem consisted of, owing to the geometric exactness of the mesh, samples of $u_e$ at points within $\eps_{\mathrm{mach}}$ distance of the true boundary of the torus. Parameter selections of $k = 2$ for $n = 1, 2$ and $k = 10$ for $n = 3, 4$ were made, to demonstrate convergence behavior within available resources. A boundary integral method~\cite{faria2021general} with convergence order equal to the expected order of convergence of the volume integral method was used to solve the homogeneous problem, and was accelerated with FMM3D with a tolerance of $\varepsilon = 10^{-13}$. The resulting maximum pointwise errors $|u(x_{j}) - u_e(x_{j})|$ measured at any of the quadrature points $\{x_j\}_{j=1}^{N}$ in $\Omega$ are reported in \Cref{fig:poissongradlaplace}.

We turn next to the $\Wcal$ and $\Xcal$ operators; results are shown in Figures~\ref{fig:WX_VIO}. The function $\Psi(x) = \cos(\alpha x_1)\sin(\alpha x_2)\cos(\alpha x_3)$ satisfies $\Lcal \Psi = \mathrm{div} g$, where
\begin{equation}\label{eq:numer_exp_g}
\begin{split}
g(x) = \alpha (\sin(\alpha x_1) \sin(\alpha x_2)\cos(\alpha x_3), &-\cos(\alpha x_1) \cos(\alpha x_2) \cos(\alpha x_3),\\
&\cos(\alpha x_1), \sin(\alpha x_2) \sin(\alpha x_3)),
\end{split}
\end{equation}
Following~\cref{W:IPP2}, the reference function is defined as
\[
v_{\mathrm{ref}}(x)
  = \mu(x)\Psi(x) + \Dcal[\Psi](x) - \Scal\lsqb \Bcal_{\nu}\Psi + g\nu \rsqb(x).
\]
The reference function for testing the $\mathcal{X}$ operator is, in turn, simply $\nabla v_{\mathrm{ref}}$. For $\Wcal$ the parameter $\alpha = 2\pi$ is chosen in~\cref{eq:numer_exp_g}; for $\Xcal$, in turn, for $n = 1, 2$ we choose $\alpha = \pi/4$ while for $n = 3, 4$ we choose $\alpha = \pi$. Results for both are presented in \Cref{fig:WX_VIO} and broadly confirm the convergence analysis; here the kernel $\sfK$ is admissible of class $\sigma = 3$ (a singular volume integral). We observe some disagreement with the theory in the $n = 4$ case; from our experience these are likely due to inaccuracies in the layer potential evaluation---the $\nabla_x \Scal$ term in the regularization formula~\cref{X:IPP} is the most challenging term in all of our regularization formulas as GPDIM~\cite{faria2021general} regularizes $\nabla_x \Scal$ through coupling to a hypersingular integral.

\begin{figure}[t]
    \centering
    \begin{subfigure}[b]{0.45\textwidth}
        \includegraphics[width=\textwidth]{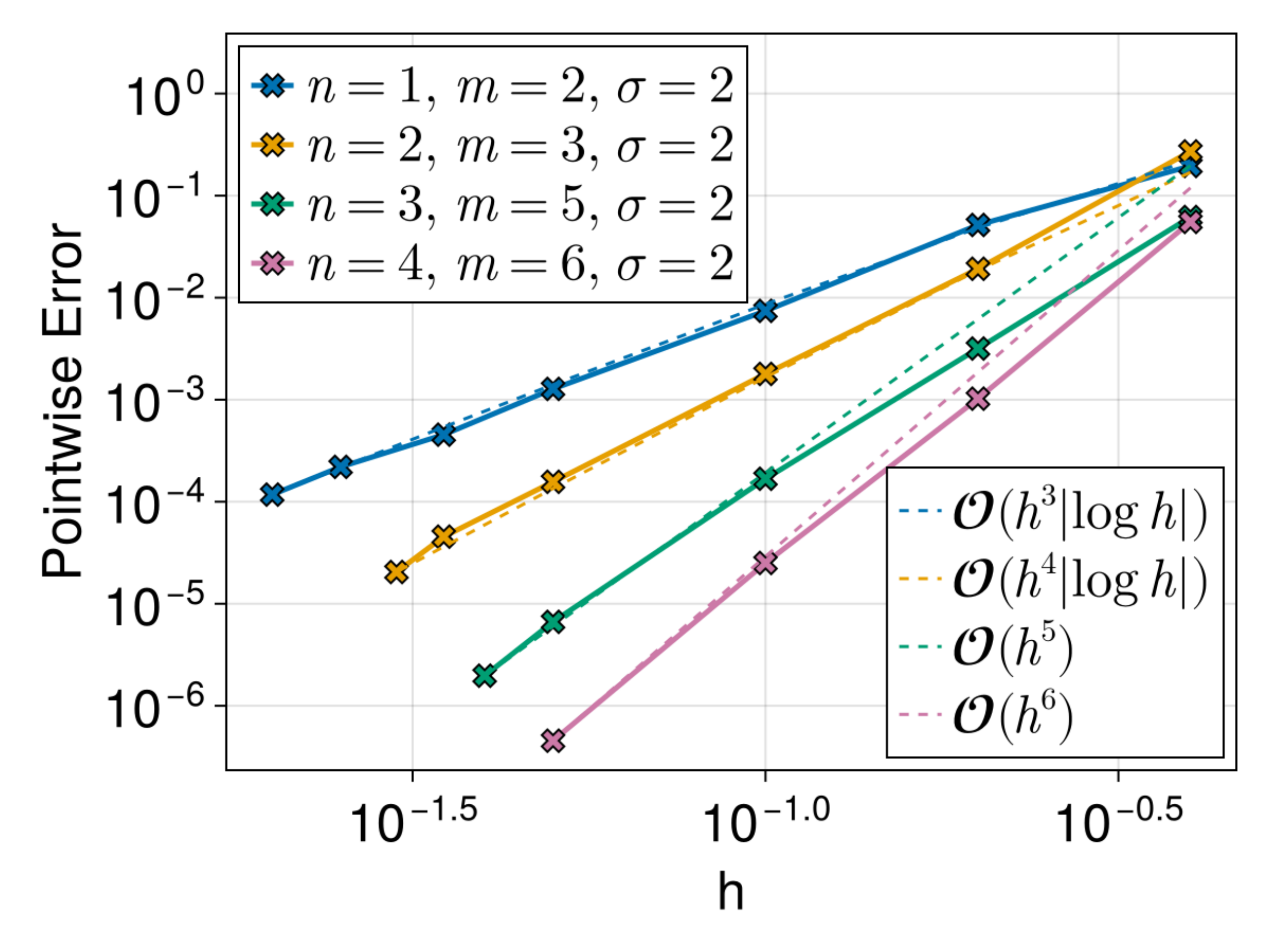}
    \end{subfigure}
    \begin{subfigure}[b]{0.45\textwidth}
        \includegraphics[width=\textwidth]{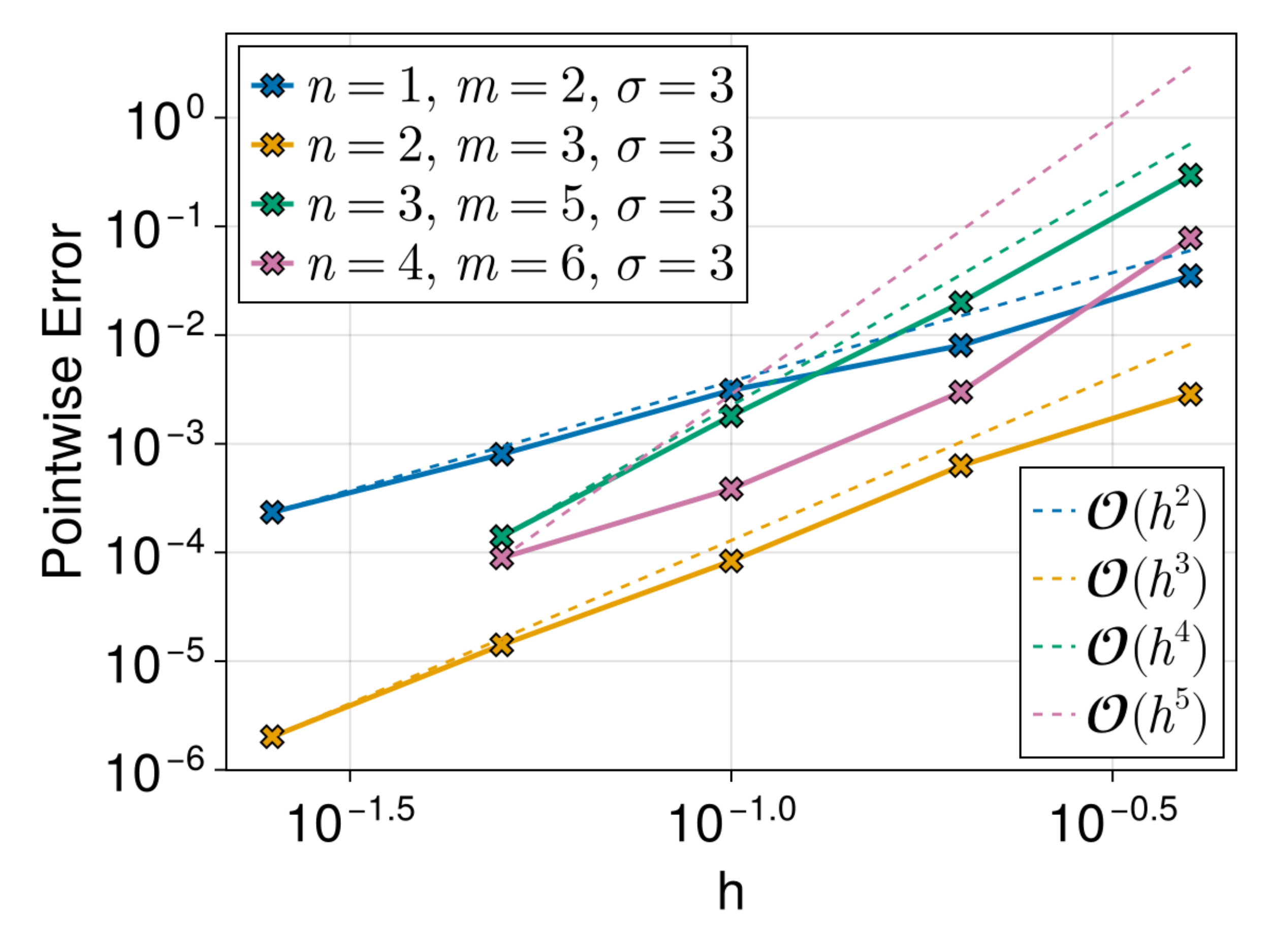}
    \end{subfigure}
    \caption{Green's identity convergence for the Laplace VIOs $\Wcal$ (left) and $\Xcal\sheq\nabla \Wcal$ (right).}
    \label{fig:WX_VIO}
\end{figure}

Vector-valued problems are tested with the Stokes equations. We choose manufactured solution $\mathbf{u}(x) = (\cos(kx_2), \sin(kx_3), \cos(kx_1))$ and $p = \sin(x_3)\cos(x_2)\cos(x_1)$ and test $\Vcal[f]$ where $f = -\Delta \mathbf{u} + \nabla p$. For $n = 1, 2$ we choose $k = 1$ while for $n = 3, 4$ we choose $k = 5$. The results of this experiment are shown in \Cref{fig:stokes} and are once again in line with expected rates of convergence.

\subsection*{Implementation Details}
The methods are implemented in the open-source integral equations library \texttt{Inti.jl}~\cite{Inti}. The curved elements of~\cite{bernardi1989optimal} are supported, but see \Cref{rem:curved} for current implementation limitations. The only nontrivial addition required to support volume potentials for a new PDE is the specification of the conormal and adjoint conormal derivatives, implementation of new polynomial solutions~\cite{anderson2024construction}, and the implementation of the relevant kernels for the layer potentials appearing in the various regularization formulas. However, a practical impediment is that the scheme for the new PDE benefits from a compatible fast summation scheme; while the kernel-independent \texttt{HMatrices.jl} is applicable it is nevertheless memory-intensive. Due to a lack of current application need, the $\Wcal$ and $\Xcal$ operators have not been implemented for vector-valued PDOs ($p > 1$).

\begin{figure}
    \centering
    \includegraphics[width=0.5\linewidth]{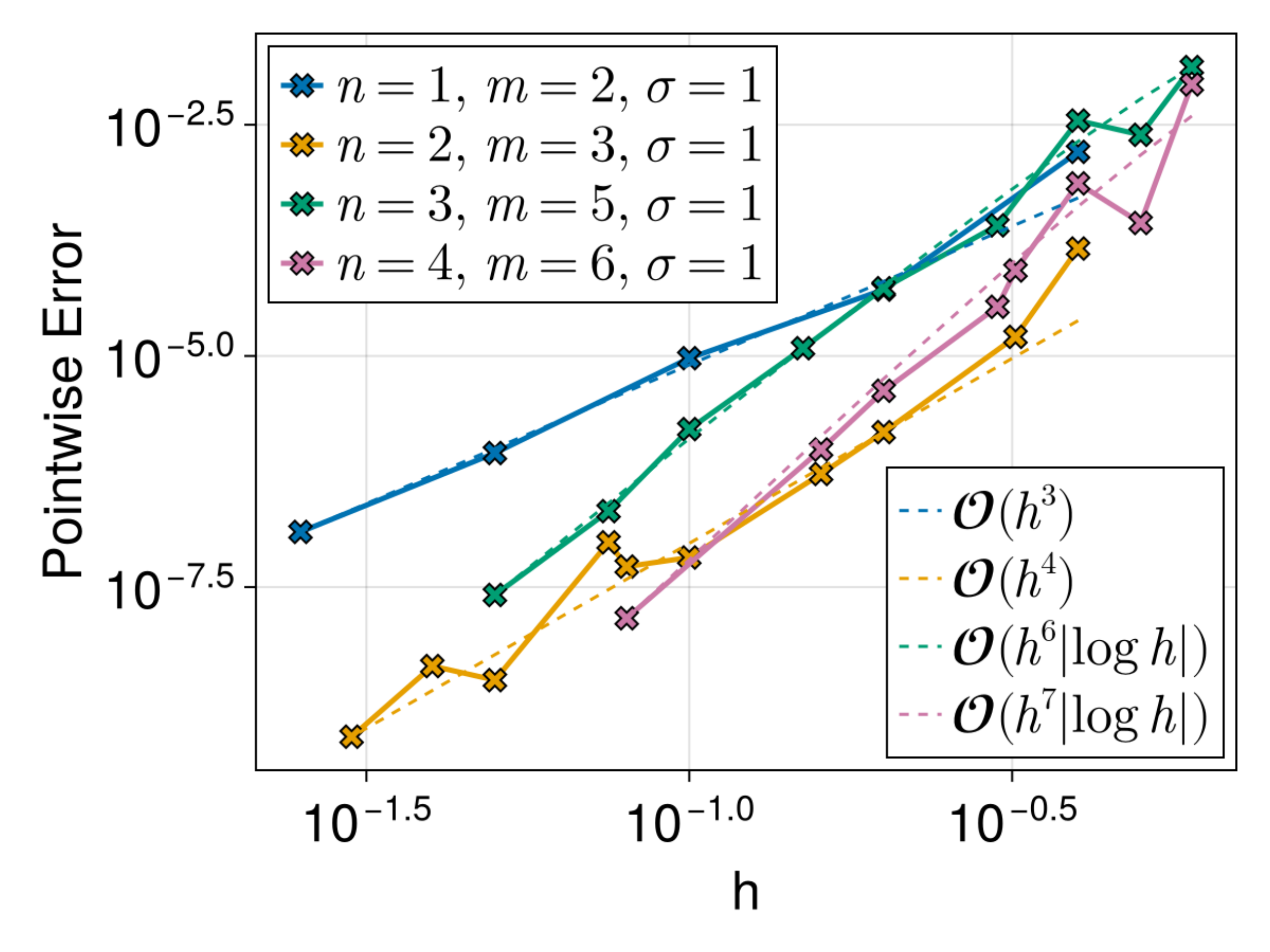}
    \caption{Green's identity convergence for Stokes VIO $\Vcal$ on a sphere of unit radius.}
    \label{fig:stokes}
\end{figure}

\section{Conclusion}\label{sec:conclusion}

We have presented and demonstrated a high-order accurate regularization strategy for a class of singular volume integral operators in three dimensions. While the work of~\cite{anderson2024fast} was in principle kernel-agnostic it established results only for Laplace and Helmholtz, whereas this work generically applies essentially to any PDO for which the Green's identities~\cref{green3} can be established. Future work will consider variants of this regularization technique tailored for problems requiring a high degree of adaptivity (some limitations and solutions in this area are discussed in~\cite{anderson2024fast}) and will consider applications in scattering and to time-dependent PDEs.

\section*{Acknowledgements}
The authors are grateful to Jesse Chan for helpful discussion on the curved finite element literature. The authors acknowledge use of AI tools in assistance with some of the associated coding and assume responsibility for all content. The first author acknowledges support from the NSF under awards DMS-2514012 and DMS-2231482. 

\bibliographystyle{siamplain}
\bibliography{References,marcbibs}
\end{document}